\documentclass[reqno]{amsart}

\usepackage{amsmath, amssymb, amsthm, nicefrac, enumerate}
\usepackage{hyperref}
\usepackage{cleveref}
\usepackage{xcolor}
\usepackage{nicematrix}
\usepackage{faktor}

\makeatletter
\DeclareRobustCommand*{\mfaktor}[3][]
{
   { \mathpalette{\mfaktor@impl@}{{#1}{#2}{#3}} }
}
\newcommand*{\mfaktor@impl@}[2]{\mfaktor@impl#1#2}
\newcommand*{\titlemfaktor@impl}[4]{
   \settoheight{\faktor@zaehlerhoehe}{\ensuremath{#1#2{#3}}}%
   \settoheight{\faktor@nennerhoehe}{\ensuremath{#1#2{#4}}}%
      \raisebox{-0.5\faktor@zaehlerhoehe}{\ensuremath{#1#2{#3}}}%
      \mkern-4mu\diagdown\mkern-5mu%
      \raisebox{0.5\faktor@nennerhoehe}{\ensuremath{#1#2{#4}}}%
}
\makeatother

\newtheorem{prop}{Proposition}
\newtheorem{thm}[prop]{Theorem}
\newtheorem{lemma}[prop]{Lemma}

\newtheorem{cor}[prop]{Corollary}

\newtheorem*{thm*}{Theorem}
\newtheorem*{alg*}{Algorithm}
\newtheorem*{lemma*}{Lemma}

\theoremstyle{definition}

\theoremstyle{remark}
\newtheorem{rmk}[prop]{Remark}
\newtheorem*{rmk*}{Remark}

\newtheorem*{notation*}{Notation}

\theoremstyle{definition}

\newtheorem{defn}[prop]{Definition}
\numberwithin{prop}{section} 
\numberwithin{equation}{section}

\newcommand{\mybf}{\mathbb}

\newcommand{\bR}{\mybf{R}}

\newcommand{\cL}{\mathcal{L}}

\newcommand{\cX}{\mathcal{X}}
\newcommand{\fs}{\mathfrak{s}}
\newcommand{\fh}{\mathfrak{h}}
\newcommand{\fg}{\mathfrak{g}}
\newcommand{\fn}{\mathfrak{n}}
\newcommand{\fz}{\mathfrak{z}}
\newcommand{\fa}{\mathfrak{a}}

\newcommand{\cZ}{\mathcal{Z}}
\newcommand{\cY}{\mathcal{Y}}
\newcommand{\cW}{\mathcal{W}}

\newcommand{\diag}[1]{\operatorname{diag}#1}
\newcommand{\Scal}[1]{\operatorname{Scal}#1}
\newcommand{\Aut}[1]{\operatorname{Aut}#1}
\newcommand{\ad}[1]{\operatorname{ad}#1}
\newcommand{\tr}[1]{\operatorname{tr}#1}
\newcommand{\abs}[1]{\left\lvert#1\right\rvert}

\newcommand{\Ric}[1]{\operatorname{Ric}#1}
\newcommand{\mRic}[1]{\operatorname{Ric^m_X}#1}

\newcommand{\Lxg}[1]{\operatorname{\mathcal{L}_X g}#1}

\title[On Nilpotent and Solvable Quasi-Einstein Manifolds ]{On Nilpotent and Solvable Quasi-Einstein Manifolds}
 \author[Nazia Valiyakath]{Nazia Valiyakath}
 \address{Department of Mathematics\\ Syracuse University, Syracuse, NY 13244}
 \email{naziavaliyakath@gmail.com}
\date{\today}

\begin{document}

\begin{abstract}
In this paper, we investigate nilpotent and unimodular solvable Lie groups that admit quasi-Einstein metrics $(M,g,X)$ with $X$ a left-invariant vector field, which we call \emph{totally left-invariant quasi-Einstein metrics}. We give a complete classification of nilpotent Lie groups admitting such metrics, proving that this occurs \emph{if and only if the group is isomorphic to a Heisenberg Lie group}. For unimodular solvable Lie groups $S$, we show that the existence of a non-flat totally left-invariant quasi-Einstein metric forces the center of $S$ to be one-dimensional. Furthermore, under the additional assumption that the adjoint action $\operatorname{ad}_a$ of $S$ is a normal derivation, we obtain a full classification: these groups are standard and their nilradical must be a Heisenberg Lie algebra. As an application, we prove that the only near-horizon geometries on a compact nilmanifold are $\Gamma \backslash H_{n}$, where $ H_{n}$ is $n$-dimensional Heisenberg Lie group.
\end{abstract}
\maketitle
\section{Introduction}
Determining the best metric for a given manifold is a classical problem in Riemannian geometry. This question, as well as problems arising in general relativity have spurred extensive study of the Einstein metrics and its various generalizations. A Riemannian manifold $(M,g)$ is called an Einstein manifold if its Ricci tensor satisfies $\Ric=\lambda g$ for some constant $\lambda\in\bR$. By \cite{DottiMiatello1982}, non-flat unimodular solvable Lie groups do not admit Einstein metrics. It is therefore natural to ask how close these groups are to admitting Einstein metrics. Quasi-Einstein metrics, a relatively newer concept, extend the classical framework of Einstein manifolds.
\begin{defn}
Let $X$ be a vector field of Riemannian manifold $(M,g),$ then $m$-Bakry-\'Emery tensor is defined as
\begin{equation}\label{main equation}
  \mRic:= \Ric + \frac{1}{2}\mathcal{L}_X g-\frac{1}{m}X^*\otimes X^*      
\end{equation}
where $m$ is any non-zero real number, $\mathcal{L}_X g$ is the Lie derivative of $g$ with respect to $X,$ and
\begin{gather*}
    X^*: T_pM\rightarrow \bR\\
    Y\mapsto g(X,Y).
\end{gather*}
\end{defn}
\begin{defn}
A Riemannian manifold is called an $m$-quasi-Einstein manifold if $\mRic = \lambda g$ for some $\lambda\in\bR.$ We denote such a manifold as a triplet $(M,g,X),$ and refer to $\lambda$ as the quasi-Einstein constant.
\end{defn}
 When $m=2$, the solutions are called near-horizon geometries, which are relevant in black hole theory \cite{Kunduri2013-lj}. In equation (\ref{main equation}), if $X=\nabla\phi$ where $\phi:M\longrightarrow \bR$ is a smooth function, then we say $(M,g)$ is gradient $m$-quasi-Einstein. In the gradient case, Case-Shu-Wei proved that any compact $m$-quasi-Einstein metric with constant scalar curvature is Einstein \cite{case2011rigidity}. Gradient m-quasi-Einstein metrics on non-compact homogeneous spaces are studied by Petersen-Wylie \cite{PETERSEN2022101929} and Lafuente \cite{https://doi.org/10.1112/blms/bdu103}. In the non-gradient case, there are many examples of $m$-quasi-Einstein metrics. In \cite{chen-Liang-Zhu}, Chen-Liang-Zhu proved that every compact simple Lie group except $SU(3), E_8$ and $G_2$ admits non-trivial $m$-quasi-Einstein metrics. In \cite{Lim2020locally}, Lim classified the compact locally homogeneous non-gradient quasi-Einstein 3-manifolds. In particular, the $3$-dimensional Heisenberg Lie group $H_3$ admits an $m$-quasi-Einstein metric $(H_3,g,X).$ Here $X$ is a Killing field, i.e., $\cL_X g=0$. This naturally raises the question: beyond the Heisenberg group $H_3,$ which other nilpotent or solvable Lie groups admit $m$-quasi-Einstein metrics? Before addressing this, it is important to note that the classification of solvable Lie algebras remains an ongoing area of research. Nilpotent Lie algebras have been classified primarily in low dimensions, and similarly, the classification of solvable Lie algebras is also limited in scope. 
 
By~\cite[Corollary 2.6]{Lim2020locally}, on an $m$-quasi-Einstein unimodular Lie group \( (G, g, X) \), if both \( X \) and \( g \) are left-invariant, then \( X \) is necessarily a Killing vector field. Since all the $m$-quasi-Einstein metrics we construct in this work feature a left-invariant vector field \( X \), we propose a terminology for this subclass of $m$-quasi-Einstein metrics.
\begin{defn}
An $m$-quasi-Einstein metric $(M,g,X)$ such that $g$ and $X$ are both left-invariant is called totally left-invariant quasi-Einstein.
\end{defn}
In~\cite{chen-Liang-Zhu}, Chen–Liang–Zhu showed that on compact Lie groups, any left-invariant quasi-Einstein metric is necessarily totally left-invariant. Moreover, all examples of gradient $m$-quasi-Einstein metrics on non-compact spaces obtained in~\cite{PETERSEN2022101929} also have the vector field $X$ left-invariant as well. However, in the non-compact setting, it is unknown whether these two notions are equivalent. A brief discussion of the case when $X$ is not left-invariant is included in Section~\ref{subsec B}.

In this paper, we first obtain some structural properties of unimodular solvable Lie groups admitting totally left-invariant quasi-Einstein metrics.
\begin{thm}
Let $(S,g,X)$ be a non-flat unimodular solvable Lie group with totally left-invariant quasi-Einstein metric. Then $X$ is in the center of the Lie algebra of $S,$ and $S$ has one-dimensional center.   
\end{thm}
This implies that every non-flat and totally left-invariant quasi-Einstein unimodular solvable Lie group has exactly two Ricci eigenvalues: the Ricci curvature is positive along the center and equal and negative in all other directions. Such a Lie group is not Einstein, but is, in a sense, close to being one. This theorem does not hold without the unimodularity condition. In \cite{PETERSEN2022101929}, there are examples of non-unimodular solvable groups with $m$-quasi-Einstein metric, that do not have a one-dimensional center. By applying this result along with Kirillov’s well-known lemma on nilpotent Lie groups with one-dimensional center, we obtain the classification in the nilpotent case.
\begin{thm}
A nilpotent Lie group admits a totally left-invariant quasi-Einstein metric if and only if it is isomorphic to a Heisenberg Lie group. Moreover, this metric is unique up to isometry and scaling.
\end{thm}
Let \( \mathfrak{s} \) be the Lie algebra of a solvable Lie group. We consider the orthogonal decomposition \( \mathfrak{s} = \mathfrak{a} \oplus \mathfrak{n} \), where \( \mathfrak{n} \) is the nilradical of \( \mathfrak{s} \) (i.e., the maximal nilpotent ideal). In the case where \( \operatorname{ad}_a \) is a normal derivation of \( \mathfrak{s} \) for all \( a \in \mathfrak{a} \), we obtain the following structure theorem.
\begin{thm}\label{thm S}
Let $S$ be a non-flat unimodular solvable Lie group with left-invariant metric $g$ and Lie algebra $\fs=\fa\oplus \fn$ where $\fn$ is the nilradical of $\fs$ and suppose $\ad_a$ is normal for all $a\in\fa.$ Then $(S,g,X)$ is totally left-invariant quasi-Einstein with quasi-Einstein constant $\lambda$ if and only if the following conditions are satisfied:
 \begin{enumerate}[(i)]
     \item $(N,g_1,X)$ is totally left-invariant quasi-Einstein with quasi-Einstein constant $\lambda$ where $N$ is any nilpotent Lie group with the Lie algebra $\fn$ and $g_1=g|_{\fn\times \fn},$
     \item $\fs$ is standard (i.e., $[\fa,\fa]=0$),
     \item $\fz(\fs)=\fz(\fn),$
 \item $g(a,a)=-\frac{1}{\lambda}\tr{S(\ad_a)^2}.$
 \end{enumerate}
\end{thm}
This implies that the nilradical of $\fs$ is Heisenberg Lie algebra, hence we investigate the normal derivations of Heisenberg, and obtain a classification in the solvable case.
\begin{thm}\label{thm Z}
   Let $S$ be a non-flat unimodular solvable Lie group with $\ad_a$ normal for all $a\in \fa$. Then $S$ admits a totally left-invariant quasi-Einstein metric $g$ if and only if its Lie algebra $\fs= \bR^k\ltimes_{\psi} \fh_n$ for some $0\leq k\leq s$ where $s=\frac{1}{2}(n-1),$ $\psi: \bR^k\rightarrow Der(\fh_n) $ is given by $a\mapsto \ad_a|_{\fh_n},$
 \[ \ad_a|_{\fh_n}= \begin{bmatrix}
   t_{1} &  &  & & &\\ 
   & -t_{1} &  & & &\\ 
   &  &  \ddots & & &\\ 
   &  &   & t_{s} & &\\
    &  &   & & -t_{s} & \\
     &  &   & &  & 0\\
 \end{bmatrix} \]
 with respect to the orthonormal basis $\{x_i,y_i,z|\, [x_i,y_j]=\delta_{ij}cz,\, 1\leq i,j\leq s\}$ of $\fh_n$ and $t_i\in \bR$ for all $1\leq i\leq s.$ 
\end{thm}
This classification theorem implies that in every dimension greater than 2, there exists a non-flat solvable Lie group admitting an $m$-quasi-Einstein metric.

A discrete cocompact subgroup $\Gamma$ of a Lie group $G$ is called 
a lattice in $G$. A metric $g$ on $\Gamma\backslash G$ is called 
\emph{invariant} if $\widetilde{g}= \pi^*(g)$ is a left-invariant 
metric on $G$, where $\pi:G\longrightarrow \Gamma\backslash G$ is 
the covering map. As mentioned earlier, it is not known whether every 
quasi-Einstein unimodular solvable Lie group is necessarily totally 
left-invariant quasi-Einstein. However, under the additional assumption 
that the group admits a lattice, every invariant quasi-Einstein metric 
on the quotient $\Gamma\backslash G$ gives rise to a totally 
left-invariant quasi-Einstein metric on $G$, as stated in the 
following theorem.
\begin{thm}
Let $G$ be a connected Lie group with lattice $\Gamma$ and covering 
map $\pi:G\longrightarrow \Gamma\backslash G$. Let $g$ be an invariant 
metric on $\Gamma\backslash G$. Then $(\Gamma\backslash G,g,X)$ is invariant $m$-quasi-Einstein if and only if $(G,\widetilde{g},\widetilde{X})$ is 
totally left-invariant quasi-Einstein and $\widetilde{X}=\pi^*(X)$ 
is left-invariant and Killing.
\end{thm}
Combining this theorem with Theorem~\ref{thm Z}, we readily obtain a structure theorem for the quotient $\Gamma \backslash G$.

A \emph{nilmanifold} is a nilpotent Lie group quotiented by its lattice. In the nilpotent case, the result takes the following form:
\begin{thm}
The only compact nilmanifolds that admit invariant $m$-quasi-Einstein metrics are of the form $\Gamma \backslash H_n$, with $n \geq 3$.
\end{thm}
In the context of black hole physics, this result implies that the only near-horizon geometry on a compact nilmanifold is $\Gamma \backslash H_{n}.$

In fact, when $X$ is Killing, an $m$-quasi-Einstein metric is similar to an \(\eta\)-Einstein metric, which comes from Sasakian geometry. A Sasakian manifold \( M \) of dimension \( 2n + 1 \), equipped with a Sasakian structure \( \mathcal{S} = (\xi, \eta, \Phi, g) \), is said to be \(\eta\)-Einstein if the Ricci curvature tensor of the metric \( g \) satisfies the equation
\[
\operatorname{Ric}_g = \lambda g + \nu\, \eta \otimes \eta
\]
for some constants \( \lambda, \nu \in \mathbb{R} \). These metrics were first introduced and studied by Okumura~\cite{Okumura1962SomeRO}, and later referred to as \(\eta\)-Einstein by Sasaki in his 1965 lecture notes. The Heisenberg group is an example of a Sasakian manifold~\cite{boyer}.

Ricci solitons are another famous generalization of Einstein metrics, and they have been extensively studied using various techniques. A Riemannian
metric $g$ on $M$ is a Ricci soliton if 
\begin{equation}\label{soliton}
  \Ric=\lambda g-\frac{1}{2}\cL_Xg  
\end{equation}
for some $\lambda\in\bR$ and some smooth vector field $X$ on $M$. Note that if $(M,g)$ is a Ricci soliton with $X$ being a Killing field, then $(M,g)$ is an Einstein manifold. Also note that Ricci solitons can be formally considered as $m$-quasi-Einstein metric with $m=\infty.$ In \cite{lauret2010riccisolitonsolvmanifolds}, Lauret proved that any solvable Lie group that admits Ricci soliton is standard and can be obtained via a very simple construction from a nilpotent Lie group that admits a Ricci soliton. But unlike quasi-Einstein metrics, there are no locally-homogeneous compact manifolds that admit non-Einstein Ricci solitons. The Heisenberg group quotiented by a lattice does not admit any invariant Ricci soliton metric, as the soliton vector field does not descend to the quotient. However, the resulting compact quotient carries an invariant $m$-quasi-Einstein metric for all $m>0.$

\medskip
\textit{\textbf{Acknowledgment:} 
The author would like to thank her advisor, Professor William Wylie, for his continuous support and guidance. She is also grateful to Professor Michael Jablonski for carefully reading an earlier draft and providing many insightful comments. She thanks Dr. Roberto Araujo for bringing important observations to her attention, which helped strengthen the results of this paper.}

\section{Preliminaries}\label{sec 2}
We begin by considering the case in which the vector field $X$ in the quasi-Einstein triple $(M,g,X)$ is Killing. In this case, the quasi-Einstein equation acquires a simpler structure and allows for stronger conclusions about the geometry of the manifold.

\begin{rmk}\label{rmk O}
Suppose $(M^n,g,X)$ is $m$-quasi-Einstein with $X$ a Killing vector field. Then, $\cL_X g=0,$ hence the quasi-Einstein equation becomes \[ \Ric{(x,x)}=\begin{cases} 
      (\lambda+\frac{\alpha^2}{m})g(x,x) & \text{ when } X=\alpha x \text{ for some } \alpha\in \bR,\\
      
      \lambda g(x,x) &  \text{ when } x \perp X.
   \end{cases}
\] This implies that the Ricci tensor has exactly two distinct eigenvalues: one along the direction of $X$, and another, of multiplicity $n-1$, along directions orthogonal to $X$.
\end{rmk}
\subsection{Quasi-Einstein metrics on compact locally homogeneous spaces}\label{sec 2.2}
Here, we discuss quasi-Einstein metrics on compact locally homogeneous spaces obtained by quotienting Lie groups by co-compact lattices.
\begin{defn}
A Riemannian manifold $(M,g)$ is locally homogeneous if for every pair of points $x, y\in M$, there exist neighborhoods $U_x$ of $x$ and $V_y$ of $y$ such that there is an isometry $\phi$ mapping $(U_x,g|_{U_x})$ to $(V_y,g|_{V_y})$, with $\phi(x)=y.$
\end{defn}
Let $\Gamma$ be a lattice of the Lie group $G,$ then the compact quotient $\Gamma \backslash G$ is locally homogeneous. In \cite[Lemma 2.4]{Lim2020locally}, Lim proves that if $(G,g,X)$ is \(m\)-quasi-Einstein where $G$ is a connected Lie group with a lattice $\Gamma$, $g$ is left-invariant, and  $X$ is a vector field which is invariant under $\Gamma$, then $X$ is a left-invariant vector field of $G$. Using this Lemma, she proves \cite[Theorem 2.5]{Lim2020locally} which says if $(\Gamma\backslash G, g, X)$ is invariant \(m\)-quasi-Einstein, then $X$ is left-invariant and Killing. Unfortunately, there is an error in the proof of \cite[Lemma 2.4]{Lim2020locally} as is pointed out in \cite{Bahuaud2025-vb} (See remarks on page 8 of \cite{Bahuaud2025-vb} before Lemma 2.8). Note that in the three-dimensional case, \cite[Lemma 2.4]{Lim2020locally} can be proven using a Milnor frame, see \cite[Lemma 2.8]{Bahuaud2025-vb}. Therefore, the validity of main results in \cite{Lim2020locally} is still correct despite the error in the proof. In fact, we give a different approach to justifying that $X$ is left-invariant for a quasi-Einstein Lie group with a lattice, in all dimensions. We use the following theorem from \cite[Theorem 1.1 a)]{Costa-Filho2024-dk} (also see \cite{Ghosh2020mQuasiEinsteinMS} and \cite{cochran2024killingfieldscompactmquasieinstein}).
\begin{thm}\label{Eric}
Let $M$ be closed (i.e. compact and without boundary) and satisfies $\mRic=\lambda g.$ Then
$(M, g)$ has constant scalar curvature if and only if $X$ is Killing.
\end{thm}
Using this we prove the following theorem.
\begin{thm}\label{thm A}
Let $G$ be a connected Lie group with lattice $\Gamma$ and covering 
map $\pi:G\longrightarrow \Gamma\backslash G$. Let $g$ be an invariant 
metric on $\Gamma\backslash G,$ such that $\tilde{g}=\pi^*(g)$. Then $(\Gamma\backslash G,g,X)$ is invariant $m$-quasi-Einstein if and only if $(G,\tilde{g},\tilde{X})$ is 
totally left-invariant quasi-Einstein and $\tilde{X}=\pi^*(X)$ 
is left-invariant and Killing.
\end{thm}
\begin{proof}
Assume $(\Gamma\backslash G,g,X)$ is invariant $m$-quasi-Einstein.  
Since $\Gamma\backslash G$ is a locally homogeneous space, it has constant scalar curvature. By Theorem \ref{Eric}, the vector field $X$ is Killing, so
\[
\mathcal{L}_X g = 0.
\]

Since $\pi$ is a local isometry, it follows that
\[
\mathcal{L}_{\tilde X} \tilde g = 0,
\]
so $\tilde X$ is Killing on $G$. 
 This implies that, on $\Gamma\backslash G,$ 
\[ \Ric{(x,x)}=\begin{cases} 
      (\lambda+\frac{\alpha^2}{m})g(x,x) & \text{ when } X=\alpha x \text{ for some } \alpha\in \bR,\\
      
      \lambda g(x,x) &  \text{ when } x \perp X.
   \end{cases}
\]
This equation lifts under $\pi,$ which implies that $(G,\Tilde{g},\Tilde{X}) $ is \(m\)-quasi-Einstein. Moreover, we get that $\Tilde{X}$ is an eigenvector of the Ricci tensor on $(G,\Tilde{g}).$ Hence the span of $\Tilde{X}$ is a one-dimensional eigenspace of Ricci tensor. But $\Tilde{g}$ is left-invariant, so the Ricci tensor is left-invariant on $ G$ as well, hence the Ricci eigenspaces are left-invariant, which implies that for any $h,p\in  G,$
\begin{equation}\label{eqn:ZZZ}
    dL_{h}(\tilde X_p)=\alpha \, \tilde X_{hp}
\end{equation}
for some $\alpha \in \mathbb{R}$. Moreover, since $(G,\tilde g,\tilde X)$ is $m$-quasi-Einstein and $\tilde X$ is Killing, taking the trace of the quasi-Einstein equation yields
\[
\Scal = \frac{1}{m}|\tilde X|^2 + n\lambda,
\]
where $n = \dim(G)$. Since $G$ is homogeneous, $\Scal$ is constant, and hence $|\tilde X|$ is also constant. Therefore,
\begin{equation*}
|\alpha| \, |\tilde X_{hp}|
= |dL_h(\tilde X_p)|
= |\tilde X_p|
= |\tilde X_{hp}|.
\end{equation*}
Here, the first equality follows from \eqref{eqn:ZZZ}, the second from the fact that $\tilde g$ is left-invariant (so $L_h$ is an isometry), and the third from the assumption that $|\tilde X|$ is constant. 

Thus, $|\alpha| = 1$, i.e., $\alpha = \pm 1$. Since $G$ is connected, it follows that $\alpha = 1$, and hence $\tilde X$ is left-invariant.

($\Longleftarrow $) Suppose $(G,\Tilde{g},\Tilde{X}) $ is \(m\)-quasi-Einstein with $\Tilde{X}, \Tilde{g} $ left-invariant. Since $\pi$ is a local isometry, for every point $h\in G$ there is a neighborhood $U$ such that $\pi^{-1}|_U$ is an isometry. Hence $(\Gamma\backslash G,g,X)$ is invariant \(m\)-quasi-Einstein as well.
\end{proof}
The above theorem implies that, once we have examples of Lie groups admitting totally left-invariant quasi-Einstein metrics and lattices, we can construct compact manifolds that admit invariant \(m\)-quasi-Einstein metrics, in contrast to the case of Ricci solitons.

\subsection{Quasi-Einstein metrics on unimodular Lie groups}

Here we discuss how the $m$-quasi-Einstein equation simplifies when 
the Lie group is unimodular.

\begin{defn}
A Lie group $G$ is called unimodular if its left-invariant 
Haar measure is also right-invariant.
\end{defn}

We also recall the definition of the adjoint operator.

\begin{defn}
Let $(\fg,[\cdot,\cdot])$ be a Lie algebra. For any $x\in\fg,$ the adjoint operator is the linear map
\[
\ad_x: \fg\rightarrow \fg, \qquad y\mapsto [x,y].
\]
\end{defn}

A connected Lie group is unimodular if and only if $\tr(\ad_x)=0$ 
for every $x$ in its Lie algebra. Also, every Lie group 
admitting a discrete lattice is unimodular.

It turns out that if $(G,g,X)$ is quasi-Einstein and $G$ is 
unimodular, then in certain situations the vector field $X$ is forced 
to be Killing. We use the following result from \cite{Lim2020locally}.

\begin{thm}\cite[Corollary 2.6]{Lim2020locally}\label{Lim}
If $G$ is a unimodular Lie group and $(G,g,X)$ is $m$-quasi-Einstein 
with left-invariant metric $g$ and left-invariant vector field $X,$ 
then $X$ is a Killing field.    
\end{thm}

In our language, this says that if $G$ is unimodular and $(G,g,X)$ 
is totally left-invariant quasi-Einstein, then $X$ is Killing. Note 
that the unimodularity assumption is necessary: by 
\cite[Proposition~7.6]{PETERSEN2022101929}, non-unimodular groups 
with Lie algebra $\bR\xi \ltimes \fh$ admit $m$-quasi-Einstein 
metrics with $\lambda=0$, provided $\fh$ is Abelian, $\ad_{\xi}$ is 
a normal derivation, and 
\[
\tr(S^2)=-\frac{(\tr S)^2}{m},
\]
where $S=\frac{1}{2}(\ad_{\xi}+\ad_{\xi}^t)$. In this case, $X$ and 
$g$ are left-invariant but $X$ is not Killing.

\section{Classification of quasi-Einstein nilpotent Lie groups}\label{sec nilpotent}

Let $(N,g)$ be a nilpotent Lie group with Lie algebra $\fn.$ Since $g$ induces an inner product on $\fn,$ we view $(\fn,g)$ as a metric Lie algebra.

For every $x\in \fn,$ the adjoint operator $\ad_x$ is nilpotent and hence $\tr{(\ad_x)}=0$ which implies that every nilpotent Lie group is unimodular. In order to find the nilpotent Lie groups that admit totally left-invariant quasi-Einstein metric, we first find the left-invariant and Killing vector fields.
\begin{prop} \label{killing}
A vector field $X$ on a nilpotent metric Lie algebra $(\mathfrak{n},g)$ is Killing if and only if $X \in \mathfrak{z}(\mathfrak{n})$, the center of $\mathfrak{n}$.
\end{prop}
\begin{proof}
Let $\{e_1,...e_n\}$ be an orthonormal basis of $\fn$ and $X\in\fn$. Then
\begin{gather*}
  \Lxg{(e_i,e_j)}= \nabla_X(g(e_i,e_j))-g(\cL_X e_i,e_j)-g(e_i,\cL_X e_j)  
\end{gather*}
Since $g$ is left-invariant, the first term is zero and since $\cL_Xe_i= [X, e_i],$ we have 
\begin{align*}
\Lxg(e_i,e_j)
&= -g([X,\,e_i],e_j)-g(e_i,[X,\,e_j])\\
&= -g(\operatorname{ad}_X e_i,e_j)-g(e_i,\operatorname{ad}_X e_j)\\
&= -g((\operatorname{ad}_X + \operatorname{ad}_X^t)e_i,e_j)
\end{align*}
for every $i,j$. Hence $X$ is Killing if and only if $\ad_X$ is skew-symmetric. But $\fn$ is nilpotent, so $\ad_X^k=0$ for some $k.$ A skew-symmetric nilpotent operator must vanish. Hence $\ad_X=0,$ which happens if and only if $X$ is in $\fz(\fn).$ 
\end{proof}
\begin{rmk}\label{rmk A}
In general, $x$ is a left-invariant and Killing vector field of a Lie group $(G,g)$ with the Lie algebra $\fg$ if and only if $\ad_x$ is skew-symmetric. If $g$ is bi-invariant, $\ad_x$ is skew-symmetric for any $x\in \fg$ and hence every left-invariant vector field of $G$ is Killing.   
\end{rmk}
From \cite[Equation 8]{MR1825405}, for any nilpotent Lie algebra $\fn,$
\begin{equation}\label{nilpotent ricci formula}
\Ric{(e_i,e_j)}= -\frac{1}{2}\sum_{k,l} g([e_i,e_k],e_l)g([e_j,e_k],e_l)
+\frac{1}{4} \sum_{k,l} g([e_k,e_l],e_i)g([e_k,e_l],e_j),
\end{equation}
where $e_i$'s form an orthonormal basis of $\fn.$ Several conclusions follow easily from the Ricci formula:
\begin{rmk}\label{formula}
By equation \eqref{nilpotent ricci formula}, the center of $\fn$ has non-negative Ricci curvature, and $[\fn,\fn]^{\perp}$, the orthogonal complement of the commutator, has non-positive Ricci curvature. In fact, \cite{MILNOR1976293} proves that for any left-invariant metric on a non-Abelian nilpotent Lie group, there exists a direction of strictly negative Ricci curvature and a direction of strictly positive Ricci curvature. Moreover, the vector with strictly negative Ricci curvature lies in $[\fn,\fn]^{\perp}$, while the vector with strictly positive curvature lies in $\fz(\fn)\cap [\fn,\fn]$. 
Note that $\fn \neq [\fn,\fn]$ by the nilpotency of $\fn$, hence $[\fn,\fn]^{\perp}$ is always non-empty. Furthermore, since $[\fn,\fn]$ is an ideal of $\fn$, the intersection $\fz(\fn) \cap [\fn,\fn]$ is also non-empty, as every ideal in a nilpotent Lie algebra intersects the center. 
\end{rmk}
The next lemma will imply that the Abelian Lie groups are the only totally left-invariant quasi-Einstein nilpotent Lie group with $\lambda=0.$ 
\begin{lemma}\label{zero eigenvalue}
Let $N$ be a nilpotent Lie group with totally left-invariant quasi-Einstein metric. Then $\lambda=0$ if and only if $N$ is Abelian.
\end{lemma}
\begin{proof}
Let $e_1,...,e_n$ be an orthonormal basis of $\fn.$ Since $X$ is left-invariant, we can assume that $X=\alpha e_1$  for some $\alpha\in \bR.$ Suppose $\lambda = 0$. Then $\Ric(e_i, e_i) = 0$ for all $2 \leq i \leq n$, since $N$ is totally left-invariant quasi-Einstein. Regardless of the value of $\Ric(e_1, e_1)$, $\fn$ either has no vector with positive Ricci curvature, or no vector with negative Ricci curvature, or no vectors with nonzero Ricci curvature at all. This contradicts Remark \ref{formula}. Hence $N$ is Abelian. The proof of the other direction is trivial. Since $N$ is Abelian, it is Ricci flat and hence $\lambda=0,$ with $X=0.$ 
\end{proof}
The next theorem provides some information about non-Abelian $m$-quasi-Einstein nilpotent Lie groups.
\begin{thm}\label{Thm C}
Suppose $(N,g,X)$ is a non-Abelian nilpotent Lie group that admits a totally left-invariant quasi-Einstein metric. Then
\begin{enumerate}[i)]
    \item The Ricci curvature of $N$ has two eigenvalues, one with multiplicity one in the direction of central element and another with multiplicity $\dim N -1.$

         \item $\lambda< 0,$ 
           \item $N$ has one-dimensional center,
         \item $\fz(\fn)\subset [\fn,\fn].$
      
\end{enumerate}
\end{thm}   
\begin{proof}
Proof of $i):$ Since $X$ is Killing, Proposition \ref{killing} implies that $X\in\fz(\fn)$. Also, being quasi-Einstein implies that \begin{equation}\label{eqn O}
\Ric{(x,x)}=\begin{cases} 
      (\lambda+\frac{\alpha^2}{m})g(x,x) & \text{ when } X=\alpha x \text{ for some } \alpha\in \bR,\\
      
      \lambda g(x,x) &  \text{ when }x\perp X
   \end{cases}
\end{equation}
as in Remark \ref{rmk O}.

To prove $ii),$ from Remark \ref{formula}, there exists a unit vector $x\in [\fn,\fn]^{\perp}$ such that $\Ric{(x,x)}<0.$ Suppose $x\perp X,$ then $i)$ implies that  $\lambda=\Ric(x,x)<0,$ hence we are done. So assume that $x=\alpha X + y$ for some $\alpha\in\bR,$ and $y\perp X.$ Then
\begin{align}\label{eqn Y}
    \Ric{(x,x)}
    &= \Ric{(\alpha X + y,\alpha X + y)}\\
    &=\alpha^2\Ric{(X,X)} + 2\alpha \Ric{(X,y)}+  \Ric{(y,y)}.\nonumber
\end{align}
Since $X$ is an eigenvector of the Ricci tensor, $\Ric{(X,y)}=0$ as $g(X,y)=0.$ Also, $i)$ implies that $\Ric{(y,y)}=\lambda g(y,y).$ So equation (\ref{eqn Y}) becomes 
\begin{equation*}
    \Ric{(y,y)}= Ric{(x,x)}-\alpha^2\Ric{(X,X)}.
\end{equation*}
But $\Ric{(x,x)}$ is given to be strictly negative. On the other hand, since $X$ is Killing, Proposition~\ref{killing} implies that $X$ lies in the center, and hence $\Ric(X,X)\geq 0.$ Therefore, $\Ric{(y,y)}=\lambda g(y,y) <0,$ which forces $\lambda<0.$

To prove $iii),$ from Proposition \ref{killing}, we have $X\in\fz(\fn).$ Let $y\in \fz(\fn)$ be a unit vector perpendicular to $X.$ Then $i)$ implies that $\Ric{(y, y)}=\lambda.$ But $\lambda<0,$ by $ii).$ This is a contradiction: since $y\in \fz(\fn),$ Remark \ref{formula} implies that $\Ric{(y,y)}\geq 0.$ Hence there are no vectors in $\fz(\fn)$ that are perpendicular to $X.$

To prove $iv),$ we know that  $\fz(\fn)$ is one-dimensional, and $\fz(\fn)\cap [\fn,\fn]\neq 0.$  Thus $\fz(\fn)\subset [\fn,\fn].$
 \end{proof}
\begin{rmk}
Since $X$ lies in the center, it has non-negative Ricci curvature. 
Hence
\[
\Ric(X,X) = \Bigl(\lambda + \tfrac{|X|^2}{m}\Bigr)|X|^2 \geq 0.
\]
Since $\lambda < 0$, this forces $m$ to be strictly positive.
\end{rmk}
This theorem gives us a class of examples of \(m\)-quasi-Einstein manifolds, the Heisenberg Lie groups. 
\begin{defn}
The Heisenberg group $H_n$ is the group of $(s+2)\times (s+2)$ matrices having the form
\[ \begin{pmatrix} 1 & \textbf{a} & c \\ 
\textbf{0} & I_s & \textbf{b}   \\ 
0 &  \textbf{0}   & 1 \end{pmatrix}\]
where  $s=\frac{n-1}{2}, \textbf{a}$ is an $s$-row vector, \textbf{b} is an $s$-column vector, and $I_s$ is the $s$-dimensional identity matrix. The corresponding Lie algebra, called Heisenberg Lie algebra, is the $n$-dimensional Lie algebra, spanned by $\{z,x_1,...,x_s,y_1,...y_s\}$ with the Lie brackets $[x_i,y_i]=z$ for all $1\leq i\leq s$ and every other Lie brackets zero. We will denote it as $\fh_{n}.$ 
\end{defn}
\begin{rmk}\label{rmk Z8}
 By \cite[Prop. 3.1]{GOZE2000615}, every two-step nilpotent Lie algebra with one-dimensional center is isomorphic to a Heisenberg Lie algebra.   
\end{rmk}
\begin{thm}\label{heisenberg}
The Heisenberg Lie group in every dimension admits a totally left-invariant quasi-Einstein metric. Moreover, such a metric is unique up to isometry and scaling.
\end{thm}
\begin{proof}
Let
\begin{equation}\label{basis}
    \{x_i, y_i, z : 1 \leq i \leq s\}
\end{equation}
be an orthonormal basis of $\mathfrak{h}_{2s+1}$ such that $[x_i, y_i] = z$. Let 
$\sigma = (\sigma_1, \ldots, \sigma_{s-1})$ be an $(s-1)$-tuple of positive real numbers 
with $\sigma_1 \geq \cdots \geq \sigma_{s-1} \geq 1$, and let $\lambda' > 0$. 

By \cite[Theorem~3.1]{VUKMIROVIC201572}, up to isometry and scaling, any left-invariant 
Riemannian metric on $H_{2s+1}$ takes the form
\[
g^{\sigma}_{\lambda',1} = \mathrm{diag}\!\left(
\sigma_1,\, \sigma_1,\, \ldots,\, \sigma_{s-1},\, \sigma_{s-1},\, 1,\, 1,\, \lambda'
\right)
\]
in basis \eqref{basis}. By \cite[Theorem~4.1]{VUKMIROVIC201572}, the Ricci curvature 
in the same basis is
\[
\rho = -\frac{\lambda'}{2}\,\mathrm{diag}\!\left(
\frac{1}{\sigma_1},\, \frac{1}{\sigma_1},\, \ldots,\,
\frac{1}{\sigma_{s-1}},\, \frac{1}{\sigma_{s-1}},\,
1,\, 1,\, -\lambda'\!\left(1 + \mathcal{H}(\sigma)\right)
\right),
\]
where $\mathcal{H}(\sigma) = \sum_{i=1}^{s-1} \dfrac{1}{\sigma_i}$. Therefore, by 
Theorem~\ref{Thm C}, the manifold $(H_{2s+1}, g^{\sigma}_{\lambda',1})$ is totally left-invariant quasi-Einstein if and only if $\sigma_i = 1$ for all 
$1 \leq i \leq s-1$.
\end{proof}
To prove that Heisenberg is the only non-Abelian totally left-invariant quasi-Einstein nilpotent group, we use the Kirillov Lemma \cite[Lemma~4.1]{Kir62}.
\begin{thm}\label{kirillov lemma}(Kirillov Lemma)
If $\fg$ is a nilpotent Lie algebra with one-dimensional center $\cZ$, then there exists a decomposition $\fg=\cX + \cY+\cZ+ \cW$ 
 in which the vector spaces $\cX, \cY$ and $\cZ$ are one-dimensional, and whose respective basis vectors $x,y,z$ may be chosen such that
$$[x,y]=z, [y,w]=0   \text{ for } w\in \cW.$$ Moreover, $\cY\oplus \cZ\oplus \cW$ is an ideal of $\fg.$
\end{thm}
The Kirillov Lemma has many important consequences in the theory of nilpotent groups. 
For instance, it implies that every nilpotent Lie algebra with a one-dimensional center 
contains a subalgebra isomorphic to $\mathfrak{h}_3$. 
We apply the Kirillov Lemma in conjunction with Theorem~\ref{Thm C}(i) to obtain the following result.
\begin{prop}\label{nilpotent main theorem}
Suppose $N$ is a non-Abelian nilpotent Lie group with totally left-invariant quasi-Einstein metric. Then 
\begin{enumerate}[i)]
    \item the Lie algebra $\fn$ of $N$ has the following structure:
 $\fn=\fz(\fn)\oplus\cX\oplus \cY\oplus  \cW$ and it admits an orthonormal basis $x,y,z,w_1,..,w_t$ such that $x,y,z$ spans $\cX,\cY,\fz(\fn)$ respectively and $w_1,...,w_t$ spans $\cW.$ Moreover,
\begin{equation*}
 [x,y]=c z, \,\,c\in \bR,
\end{equation*}
\begin{equation*}
   [w_i,w_j]\in \cW\oplus \fz(\fn) \hspace{2em} \forall\, 1\leq i,j \leq t  
\end{equation*}
and all other Lie brackets between the basis vectors are zero except for the anti-symmetric relations.
\item $\lambda=-\frac{1}{2}c^2.$
\end{enumerate}
\end{prop}
\begin{proof}
First we prove that $N$ admits an orthonormal basis that satisfies the properties of Kirillov Lemma. Since $N$ is totally left-invariant quasi-Einstein, it has one-dimensional center $\cZ$. So, by Kirillov lemma, there exists $x,y,z\in \fn$ such that $[x,y]=z,$ and there exist a vector space decomposition $\fn=\cX + \cY+\cZ+ \cW$ as in the Kirillov Lemma. 

Let $\cZ\oplus \cZ^{\perp}$ be an orthogonal decomposition of $\fn$ with respect to the \(m\)-quasi-Einstein metric $g.$ Suppose $\cZ^{\perp}=\Bar{\cX}+\Bar{\cY}+\Bar{\cW}$ is spanned by the vectors \[\Bar{x}=x+az,\quad\Bar{y}=y+bz,\quad \Bar{w_i}=w_i+c_i z\] 
where $x\in \cX, y\in \cY,$ and $w_i\in\cW$ for every $1\leq i\leq t.$ Since $z$ is in the center of $\fn,$ it is easy to check that $[\Bar{x},\Bar{y}]=z,[\Bar{y},\Bar{w_i}]=0$ for all $i,$ and $\cZ+\Bar{\cY}+\Bar{\cW}$ is an ideal of $\fn,$ which are all the properties of Kirillov decomposition. So we can assume that $\Bar{\cX}=\cX,\Bar{\cY}=\cY$ and $\Bar{w_i}=w_i$ for all $i.$ Now in $\cY+\cW,$ pick $\Bar{\Bar{\cW}}$ such that it is the orthogonal complement of $\cY.$ Let $\Bar{\Bar{\cW}}$ be spanned by the orthonormal basis $\Bar{\Bar{w_i}}= w_i + d_i y$ where $1\leq i\leq t.$ Then $[y,\Bar{\Bar{w_i}}]=0$ since $[y,w_i]=0$ by Kirillov Lemma. Now, pick $\Bar{\Bar{\cX}}$ such that it is orthogonal to $\cY+\cW$ inside $\cX+\cY+\cW.$ Then $\Bar{\Bar{\cX}}$ is spanned by some vector $\Bar{\Bar{x}}=x+ey+fw,$ where $w\in \cW.$ Then $[\Bar{\Bar{x}},y]=z.$ Also, it is easy to check that $ \cY+ \cZ +\Bar{\Bar{\cW}} $ is an ideal of $\fn.$ So the decomposition $\fn= \Bar{\Bar{\cX}}+ \cY+\cZ + \Bar{\Bar{\cW}} $ is orthogonal and the orthogonal basis $\{\Bar{\Bar{x}}, y,z, \Bar{\Bar{w_1}},...,\Bar{\Bar{w_t}}\}$ spans $\fn.$

Now, let $\{x,y,z,w_1,...,w_t\}$ be an orthonormal basis of $\fn$ as in the Kirillov Lemma, then we have that $[x,y]=cz$ for some $c\in\bR$. From the Kirillov Lemma we have that $\cY\oplus \cZ\oplus \cW$ is an ideal, so $[u,v]\notin \cX$ for any $u,v\in \cY\oplus \cZ\oplus \cW$ and also $[x,u]\in \cY\oplus \cZ\oplus \cW$ for any $x\in \cX,u\in \fn,$ hence $\cX\cap [\fn,\fn]=0.$ So from the Ricci curvature formula \eqref{nilpotent ricci formula}, the positive term in $\Ric(x,x)$ is zero, and hence
\begin{align}
\Ric(x,x) &= -\frac{1}{2}g([y,x],z)^2-\frac{1}{2}\sum_{i=1}^t\|[x,w_i]\|^2
\label{eqn U} \\
\Ric(y,y) &= -\frac{1}{2}g([y,x],z)^2+\frac{1}{2}\sum_{i=1}^t g([x,w_i],y)^2
+\frac{1}{4}\sum_{i,j=1}^t g([w_i,w_j],y)^2
\label{eqn P}
\end{align}
But $(N,g,X)$ is totally left-invariant quasi-Einstein. So Theorem \ref{Thm C}$\,i)$ implies that $\Ric{(x,x)}=\Ric{(y,y)}=\Ric{(w_i,w_i)}$ for every $1\leq i\leq t.$ So, equating (\ref{eqn U}) and (\ref{eqn P}) we get
\begin{gather*}
    ||[x,w_i]||^2=g([x,w_i],y)^2=g([w_i,w_j],y)^2=0
\end{gather*}
 for every $1\leq i,j\leq t.$ Hence
$[\cX,\cW]=0$ and $[\cW,\cW]\subset \cW+\fz(\fn),$ proves $i).$ Also,
\begin{equation*}
    \lambda=\Ric(x,x)=-\frac{1}{2}g([y,x],z)^2=-\frac{1}{2}c^2
\end{equation*}
proves $ii).$
\end{proof}
Using this, we can now obtain a complete classification of totally left-invariant quasi-Einstein nilpotent Lie groups.
\begin{thm}\label{classification}
The Heisenberg Lie group is the only non-Abelian nilpotent Lie group that admits a totally left-invariant quasi-Einstein metric, up to isomorphism.  
\end{thm}
\begin{proof}
We prove the claim by induction on the dimension of the Lie algebra. 
Let $(\fn,g)$ be the metric Lie algebra of a non-Abelian nilpotent Lie group that admits a totally left-invariant quasi-Einstein metric. 
If $\dim(\fn)\leq 3$, then $\fn$ must be $\fh_3$ since Heisenberg is the only non-Abelian nilpotent Lie group in $3$-dimension. Assume that the claim holds when $\dim(\fn)\leq k$.

Let $(N',g,X)$ be a nilpotent Lie group admitting a totally left-invariant quasi-Einstein metric, with Lie algebra $\fn'$ and $\dim(\fn')=k+2$. 
By Proposition~\ref{nilpotent main theorem}, $\fn'$ admits an orthonormal basis $\{x,y,z,w_1,\dots,w_{k-1}\}$ where $\{w_i\}$ spans a subspace $\cW$ and
\begin{equation*}
    [x,y]=cz,\quad [w_i,w_j]\in \cW \oplus \fz(\fn'),
\end{equation*}
with all other Lie brackets between basis vectors equal to zero, except the anti-symmetric relations.
Since $(N',g)$ admits a totally left-invariant quasi-Einstein metric, by Theorem \ref{Thm C}.(i) we have 
\begin{align}
\Ric(x,x)&=\Ric(y,y)=-\tfrac{1}{2}c^2=\lambda, \nonumber\\
\Ric(w_k,w_k)
&=-\tfrac{1}{2}\sum_{i,j} g([w_k,w_i],w_j)^2
  -\tfrac{1}{2}\sum_i g([w_k,w_i],z)^2  \nonumber\\
&\quad+\tfrac{1}{4}\sum_{i,j} g([w_i,w_j],w_k)^2
=\lambda, \label{eqn Z1}\\
\Ric(z,w)
&=\tfrac{1}{4}\sum_{i,j} g([w_i,w_j],z)\, g([w_i,w_j],w)=0,\\
\Ric(w_k,w_l)
&=-\tfrac{1}{2}\sum_{i,j} g([w_k,w_i],w_j)\, g([w_l,w_i],w_j) \nonumber\\
&\quad-\tfrac{1}{2}\sum_i g([w_k,w_i],z)\, g([w_l,w_i],z) \nonumber\\
&\quad+\tfrac{1}{4}\sum_{i,j} g([w_i,w_j],w_k)\, g([w_i,w_j],w_l)
=0. \nonumber
\end{align}
Moreover,
\begin{align}\label{eqn Z2} 
\Ric(z,z)&=\frac{1}{2}c^2 + \frac{1}{4}\displaystyle\sum_{i,j}g([w_i,w_j],z)^2 = \lambda + \frac{\alpha ^2}{m}.\nonumber\\ \end{align} 
where $\alpha\in\bR$ such that $X=\alpha z.$ Note that $x$ and $y$ do not appear in the Ricci computation of $\Ric(w_k,w_l)$ since $[x,w]=[y,w]=0$ and $g([x,y],w)=0$ for all $w\in\cW$.

Now, consider the subalgebra $\fn=\cW \oplus \fz(\fn')$ of $\fn'$ spanned by the orthonormal basis $\{z,w_1,\dots,w_{k-1}\}.$ 
Let $(N,g|_{N})$ be the Lie group with Lie algebra $\fn.$ 

Since any subalgebra of a nilpotent Lie algebra is nilpotent (see \cite[Prop. 1.10]{MR1920389}), $\fn$ is nilpotent. 
Also, $[x,w]=[y,w]=0$ for all $w\in \cW$ implies that $\fn$ has the same center as $\fn'$, and hence $\fz(\fn')=\bR z$. 
Let $\Ric^{\fn}$ denote the Ricci curvature of $(N,g|_{N})$. 
Comparing with equations (\ref{eqn Z1})-(\ref{eqn Z2}), we see that for all $i,j$
\begin{equation}
    \Ric^{\fn}(w_i,w_j)=\Ric (w_i,w_j) =\delta_{ij}\lambda,\qquad  \Ric^{\fn}(z,w)=\Ric (z,w)=0, 
\end{equation}
and 
\begin{equation}\label{eqn Z6}
    \Ric^{\fn}(z,z)= \tfrac{1}{4}\displaystyle\sum_{i,j}g([w_i,w_j],z)^2 = \frac{\alpha ^2}{m}-c^2,
\end{equation}
where the last equality is from (\ref{eqn Z2}). The above equations show that $\Ric^{\fn}(u,v)=\Ric(u,v)$ for all $u,v\perp z.$ Let
\[m'=\frac{2m\alpha^2 }{2\alpha^2 -mc^2},\] then, using (\ref{eqn Z6}) we get,
\[\Ric^{\fn}(z,z)-\frac{\alpha^2}{m'}=  \lambda.\]
Let $X=\alpha z,$ then $(N,g|_N,X)$ satisfies 
\[\Ric^{\fn}-\tfrac{1}{m'}X^*\otimes X^*=\lambda g.\]
So $\fn$ admits a totally left invariant quasi-Einstein metric. Note that in the expression of $m',$ the denominator \[2\alpha^2 -mc^2\neq 0.\]  
Otherwise, we would have $\tfrac{\alpha^2}{m} = \tfrac{1}{2}c^2$, and by (\ref{eqn Z6}) it follows that 
\[
\Ric^{\fn}(z,z) = \lambda,
\]
which implies $\Ric^{\fn} = \lambda g$. Hence, $\fn$ would be Einstein, but then $\fn$ must be abelian by \cite{DottiMiatello1982}, forcing $\lambda = 0$. By Lemma~\ref{zero eigenvalue}, this would imply that $\fn'$ is Abelian, a contradiction. 

Since $\dim(\fn)=k$, the inductive hypothesis applies. 
Therefore, $\fn$ is isomorphic to the Heisenberg Lie algebra $\fh_{k}$. 
In particular, $[\fn,\fn]=\fz(\fn)=\bR z$, as in the Heisenberg case, and hence
\[[w_i,w_j]= c_{ij} z\] 
for some $c_{ij}\in \bR.$ 
 Hence, $\fn'$ has the Lie brackets \[[x,y]=cz, \qquad[w_i,w_j]= c_{ij} z,\] hence it is a two-step nilpotent Lie algebra with one-dimensional center. So, by Remark \ref{rmk Z8}, $\fn'$ is isomorphic to $\fh_{k+2}.$
\end{proof}
 Based on our discussion in section \ref{sec 2.2}, to obtain a classification of \(m\)-quasi-Einstein compact nilmanifolds, we need to look at the nilpotent groups that admit lattice. This is described by the following well-known theorem of Malcev \cite[Theorem 7]{MR39734} (also see \cite[Theorem 2.12]{Rag72} ):
\begin{thm}\label{rational}
 Let $N$ be a simply connected nilpotent Lie group and let $\fn$ be its Lie algebra. Then $N$ admits a lattice if and only if it admits a basis with respect to which the structure constants are rational.
\end{thm}
 The Heisenberg Lie groups admit lattices, since they possess a basis with rational structure constants. For example, the subgroup of matrices in the Heisenberg group with integer entries is a lattice in the Heisenberg Lie group. Therefore, the classification theorem we proved implies the following.
 \begin{thm}\label{thm A1}
The only compact nilmanifolds that admit invariant $m$-quasi-Einstein metrics are of the form $\Gamma \backslash H_n$, with $n \geq 3$.
\end{thm}
\begin{proof}
Suppose  $\Gamma \backslash N$ is  a compact nilmanifold that is invariant $m$-quasi-Einstein manifold, then Theorem \ref{thm A} implies that $N$ admits a totally left-invariant quasi-Einstein metric. Hence $N$ has to be Heisenberg Lie group.
\end{proof}
An $m$-quasi-Einstein manifold $(M,g,X)$, where $M$ is closed and $m = 2$, is called a \emph{near-horizon geometry} (see \cite{Kunduri2013-lj}). So, if a compact nilmanifold admits a $2$-quasi-Einstein metric, then it must be a near-horizon geometry. Combining this with Theorem \ref{thm A1}, we obtain:
\begin{cor}
  The only near-horizon geometries on a compact nilmanifold are $\Gamma \backslash H_{n}$, where $n \geq 3.$     
\end{cor}
This result appears to be new even in the case of near horizon geometries.

\section{Structure of quasi-Einstein unimodular solvable Lie groups}
We now turn to the study of quasi-Einstein solutions in a bigger class of Lie groups: the solvable Lie groups.

We consider a Lie group $(S,g)$ to be unimodular solvable with left-invariant metric $g$ and Lie algebra $\fs.$ Due to Theorem \ref{Lim}, our first objective is to identify the Lie subalgebra of Killing fields in $\fs.$ It turns out that if we assume that $S$ is totally left-invariant quasi-Einstein, then its Lie algebra of left-invariant Killing fields coincides precisely with $\fz(\fs),$ the center of $\fs.$ In contrast to the nilpotent case, this result for solvable Lie groups requires additional effort. In particular, we will make use of the Ricci curvature formula of $S.$

Since $\fs$ is solvable, the commutator $[\fs,\fs]$ is a nilpotent subalgebra of $\fs.$ Consider the orthogonal decomposition $\fs=[\fs,\fs]\oplus[\fs,\fs]^{\perp}.$ Given an orthonormal basis $e_1,...,e_n$ of $\fs,$  from \cite[Equation 3.1]{DottiMiatello1982} we have the following Ricci curvature formulae for a unimodular solvable Lie algebra:
\begin{align}
    \Ric(h,h) &= -\frac{1}{2}\tr(\ad_h^t \ad_h) + \frac{1}{2}\sum_{i<j} g([e_i,e_j], h)^2  \nonumber \\
    \Ric(f,f) &= -\frac{1}{4} \tr(\ad_f + \ad_f^t)^2,\label{ricci}  \\
    \Ric(h+f,h+f) &= \Ric(h,h) + \Ric(f,f) - \tr(\ad_h^t \ad_f), \nonumber
\end{align}
where $h\in [\fs,\fs]$ and $f\in [\fs,\fs]^{\perp}.$ We will also need \textit{Cartan's criterion}, which says that in a solvable Lie algebra, $\tr(\ad_h \ad_x)=0$ for any $h\in [\fs,\fs], x\in \fs$.

Using this, we prove a couple of Lemmas.
\begin{lemma}\label{Lemma U}
Let $S$ be a unimodular solvable Lie group with left-invariant metric $g$ and Lie algebra $\fs$. If $z\in \fz(\fs),$ then $\Ric{(z,z)}\geq 0.$   
\end{lemma}
\begin{proof}
Let \( z = h + f \), where \( h \in [\fs,\fs] \) and \( f \in [\fs,\fs]^{\perp} \). Since
\[
0 = \ad_z = \ad_{h+f} = \ad_h + \ad_f,
\]
it follows that
\[
\ad_h = -\ad_f.
\]
Taking transposes, we obtain
\begin{equation}\label{eqn A1}
\ad_h + \ad_h^t = -(\ad_f + \ad_f^t).
\end{equation}
Substituting \eqref{eqn A1} into \eqref{ricci}, and using that $\tr(\ad_h)^2=0$ due to Cartan's criterion, we obtain
\begin{equation}\label{eqn A3}
\Ric(f,f) = -\frac{1}{2}\tr(\ad_h \ad_h^t).
\end{equation}
 Compose (\ref{eqn A1}) with $\ad_h,$ take the trace and apply Cartan's criterion to get,
 \begin{equation}\label{eqn A2}
     \tr{\ad_h \ad^t_h}= -\tr{\ad_h \ad^t_f}.
 \end{equation}
Plugging (\ref{eqn A3}) and (\ref{eqn A2}) into (\ref{ricci}) gives,
\begin{equation*}
    \Ric{(h+f,h+f)}=\frac{1}{2}\sum_{i<j}g([e_i,e_j],h)^2
\end{equation*}
which is non-negative.
\end{proof}
\begin{lemma}\label{Lemma X}
 Let $\fs$ be a non-Abelian unimodular solvable Lie algebra with left-invariant metric $g$. If $X\in\fs $ is a Killing field, then $\Ric{(X,X)}\geq 0,$ where the equality holds when $X\in [\fs,\fs]^{\perp}.$
\end{lemma}
\begin{proof}
Since $X$ is Killing, Remark \ref{rmk A} implies that
\begin{equation}\label{Killing equation}
    \ad_X + \ad^t_X=0
\end{equation} on $\fs$.
Decompose $\fs$ as an orthogonal direct sum of subspaces  
\begin{gather*}
    \fs=[\fs,\fs]\oplus[\fs,\fs]^{\perp}.
\end{gather*}
Let $\{e_1,...,e_n\}$ be an orthonormal basis of $\fs$. Assume $X=h+f$ for some $h\in [\fs,\fs ],f\in [\fs,\fs]^{\perp}$. From (\ref{Killing equation}), $\tr{(\ad_X + \ad^t_X)^2}=0.$ Expanding this, and canceling terms by applying Cartan's criterion we get
\begin{equation}\label{eqn W}
    \tr(\ad_f)^2+\tr{\ad_h \ad^t_h}+\tr{\ad_f \ad^t_f}+ 2 \tr \ad_h \ad^t_f=0.
\end{equation}
But from (\ref{ricci}),
\begin{equation}\label{eqn Q}
    \Ric(X,X)=-\frac{1}{2}\tr{\ad^t_h \ad_h}+\frac{1}{2}\sum_{i<j}g([e_i,e_j],h)^2-\frac{1}{4}\tr{(\ad_f +\ad^t_f)^2}-\tr{\ad^t_h \ad_f}.
\end{equation}
But $\tr{(\ad_f +\ad^t_f)^2}=2\tr(\ad_f)^2+2\tr \ad_f\ad_f^t.$ Hence (\ref{eqn Q}) becomes 
\begin{equation}\label{eqn R}
\begin{aligned}
\Ric(X,X)={}&-\frac{1}{2}\tr(\ad^t_h \ad_h)
+\frac{1}{2}\sum_{i<j} g([e_i,e_j],h)^2
-\frac{1}{2}\tr(\ad_f)^2 \\
&-\frac{1}{2}\tr(\ad_f\ad_f^t)
-\tr(\ad^t_h \ad_f).
\end{aligned}
\end{equation}
Applying (\ref{eqn W}), this becomes
\begin{equation}\label{eqn S}
\Ric(X,X)= \frac{1}{2}\sum_{i<j}g([e_i,e_j],h)^2\geq 0. 
\end{equation} 
This right-hand side of the equality is zero either if $\fs$ is Abelian, or if $h=0.$
\end{proof}
We also require the following theorem regarding the curvature of solvable Lie groups with left-invariant metrics.
\begin{thm}\cite[Theorem 3.1]{MILNOR1976293}\label{thm T}
If the Lie group $S$ is solvable, then every left-invariant metric on $S$ is either flat, or else has strictly negative scalar curvature.   
\end{thm}
Using this theorem and the previous lemma we prove:
\begin{lemma}\label{lemma Y}
Let $S$ be a non-flat unimodular solvable Lie group admitting a totally left-invariant quasi-Einstein metric. Then $\lambda \leq 0$, with equality if and only if $S$ is flat.
\end{lemma}
\begin{proof}
We will prove the flat case first. Suppose $\lambda=0.$ Assume that $\{e_1,..,e_n\}$ is an orthonormal basis of $\fs,$ and that $X=\alpha e_1.$ Since $\lambda=0,$ we have that $\Ric{(e_i,e_i)}=0$ for all $2\leq i\leq n.$ Since $X$ is Killing, Lemma \ref{Lemma X} implies that $\Ric{(X,X)}\geq 0.$ This implies that the scalar curvature of $S$ is non-negative. Hence Theorem \ref{thm T} implies that $S$ is flat. The other direction is immediate.

Now suppose that $\lambda > 0$. Then the previous argument implies that $S$ is non-flat. Again, by Lemma~\ref{Lemma X}, we have $\Ric(X,X) \geq 0$. Hence, $\Ric \geq 0$ on $S$, which contradicts the fact that $S$ is non-flat.
\end{proof}
We require one additional tool to prove that $X$ lies in the center. In \cite[p.~260]{DottiMiatello1982}, Dotti states the following result; for completeness, we include a proof.
\begin{lemma}\label{rmk dotti} 
Suppose $\Ric{}\leq 0$ on a unimodular solvable Lie group. Then, $\Ric{(h,h)}=0$ for every $h\in [\fs,\fs].$    
\end{lemma}
\begin{proof}
Denote by $h_1,...,h_q;f_1,...,f_r$ an orthonormal basis of $[\fs,\fs],[\fs,\fs]^{\perp}$ respectively.

Note that
\begin{equation*}
  \tr{\ad^t_{h_i} \ad_{h_i}}=\sum_{j,k}g([f_j,h_i],h_k)^2+\sum_{j,k}g([h_j,h_i],h_k)^2.   
\end{equation*} Hence
\begin{align*}
\sum_{i}\Ric(h_i,h_i)
&= -\frac{1}{2}\sum_{i} \tr(\ad^t_{h_i}\ad_{h_i})
   +\frac{1}{2}\sum_{i,j,k} g([f_k,h_j],h_i)^2  \\
&\quad +\frac{1}{2}\sum_{i,j,k} g([h_k,h_j],h_i)^2
   +\frac{1}{2}\sum_{i,j,k} g([f_k,f_j],h_i)^2 \\
&= \frac{1}{2}\sum_{i,j,k} g([f_k,f_j],h_i)^2\geq 0.
\end{align*}
This, together with the hypothesis $\Ric \leq 0$, forces 
$\sum_{i}\Ric(h_i,h_i)=0$, and hence $\Ric(h_i,h_i)=0$ for all $i.$
\end{proof}

Now we are ready to prove our theorem.

\begin{thm}\label{solvable killing}
Let $(S,g,X)$ be a non-flat unimodular solvable Lie group with totally left-invariant quasi-Einstein metric. Then $X$ is in $\fz(\fs)$ and  $\dim(\fz(\fs))=1.$ 
\end{thm}

\begin{proof}
By Theorem \ref{Lim}, $X$ is Killing. So all the equations in the proof of Lemma \ref{Lemma X} hold. Assume $X=h+f$ for some $h\in [\fs,\fs ],f\in [\fs,\fs]^{\perp}$.
From equation (\ref{Killing equation}) we have that
\begin{equation*}
    \ad_h + \ad_h^t = -(\ad_f+\ad_f^t).
\end{equation*}
 Compose this with $\ad_h,$ take the trace, and apply Cartan’s criterion to get,
 \begin{equation}\label{2Ric(h,f)}
  \tr(\ad_h^t \ad_h) = -\tr(\ad_f^t  \ad_h)=2\Ric(h,f)   
 \end{equation}
 where last equality is from (\ref{ricci}). Putting (\ref{2Ric(h,f)}), and (\ref{eqn S}) in (\ref{ricci}) we get,
\begin{equation*}
  \Ric(h,h)= -\Ric(h,f)+\Ric{(X,X)}.
\end{equation*}
Bilinearity of Ricci tensor implies
\begin{equation}\label{Ric(h,X)}
 \Ric(h,X) = \Ric(X,X).    
\end{equation}
Since $S$ is totally left-invariant quasi-Einstein, for any $x\in \fs,$
\begin{equation*}
    \Ric(x,X)=\Big(\lambda+\frac{\abs{X}^2}{m}\Big)g(x,X).
\end{equation*}
This means that $X$ is an eigenvector of the Ricci tensor with eigenvalue  $\nu=\lambda+\frac{\abs{X}^2}{m}.$ Hence (\ref{Ric(h,X)}) implies that
\begin{equation*}
0=\Ric(X-h,X) = \Ric{(f,X)}=\nu g(f,X)=\nu\abs{f}^2.
\end{equation*}
This implies that either $ \nu=0$ or $f=0.$ If $\nu=0,$ Lemma \ref{Lemma X} implies that $X=f.$ Also, since $\lambda<0,$ we have that $\Ric{}\leq 0$ on $S.$ So Lemma \ref{rmk dotti} implies that $\Ric{(u,u)}=0$ for all $u\in [\fs,\fs].$ This contradicts $\lambda<0.$ So $f=0.$ Then (\ref{2Ric(h,f)}) implies that $ \tr(\ad_h^t \ad_h) = 0,$ hence $\ad_h=0$ which implies that $X=h\in \fz(\fs).$ Also, note that this implies that $X\in [\fs,\fs]$.

To prove that the center is one-dimensional, suppose $z\in \fz(\fs)$ such that $z$ is orthogonal to $X.$ By Lemma \ref{Lemma U}, we get that $\lambda=\Ric{(z,z)}\geq 0.$ Then Lemma \ref{Lemma X} implies that $\Ric\geq 0$ on $S,$ hence $S$ is flat by Theorem \ref{thm T}, which is a contradiction. So $z=0,$ so $\dim(\fz(\fs))=1.$
\end{proof}
 This is a good time to explore how Ricci solitons and \(m\)-quasi-Einstein metrics are closely related. Fix a left-invariant vector field $X$ of solvable Lie group $(S,g)$ with Lie algebra $\fs$ and $g$ left-invariant metric. We define the linear operator $F$ on $\fs$ such that $g(F(x),y)=\frac{1}{m}X^*\otimes X^*(x,y)$ for any $x,y\in\fs.$ Define $r$ to be the Ricci operator such that $\Ric{(x,y)}=g(r(x),y).$ Hence for convenience we rewrite the \(m\)-quasi Einstein  equation (\ref{main equation}):
 
A unimodular solvable Lie group $(S,g,X)$ is totally left-invariant quasi-Einstein if and only if 
\begin{equation}
 r=\lambda I+F,   
\end{equation}
where
\[F(x)= \begin{cases} 
      \frac{g(X,X)}{m}x & x\in \fz(\fs), \\
      0 & x\in \fz(\fs)^{\perp}. \\
  
   \end{cases}
\]
If $F$ were a derivation, then every totally left-invariant quasi-Einstein solvable Lie group would also be a Ricci soliton. However, this is never the case. To see this, suppose $(S,g,X)$ is a non-flat totally left-invariant quasi-Einstein solvable Lie group. Then $\fz(\fs)$ is one-dimensional, and hence $\fz(\fs)\subseteq [\fs,\fs]$. In particular, there exist vectors $x,y$ orthogonal to $z\in \fz(\fs)$ such that $[x,y]=z$. Since $F(x)=F(y)=0$ but $F([x,y])\neq 0$, we obtain a contradiction. Indeed, if $F$ were a derivation, then we would have
\[
F([x,y])=[F(x),y]+[x,F(y)]=0.
\]
Therefore, $F$ cannot be a derivation. This also rules out the use of techniques from geometric invariant theory as in \cite{lauret2010riccisolitonsolvmanifolds}.

We note that in the case of the Heisenberg Lie group $(H_{2s+1},g)$, the metric $g$ admits both an $m$-quasi-Einstein structure and a Ricci soliton structure. That is,
\[
\Ric = \lambda I + F,
\]
and
\[
\Ric = \lambda' I + D,
\]
where $\lambda'=(s+2)\lambda$ and $D=\diag(1,...,1,2)$. Hence,
\[
F=(s+1)\lambda I + D.
\]
Note that $D$ is a derivation, whereas the identity map is not a derivation of the Heisenberg Lie algebra. Thus it is consistent that $F$ itself is not a derivation.

While a complete classification of $m$-quasi-Einstein solvable Lie groups remains out of reach, we will provide one in a very special case.

For this, we first give the required preliminaries. Let $(S,g)$ be the unimodular solvable Lie group with Lie algebra $\fs.$ Consider the orthogonal decomposition $\fs=\fa\oplus \fn,$ where $\fn$ is the nilradical (maximal nilpotent ideal) of $\fs.$ 
Let $H$ be the mean curvature vector of $S$ which is the unique element such that $g(H,a)= \tr \ad_a$ for any $a\in \fa.$

Let $a_1,...,a_q,x_1,...,x_n$ be an orthogonal basis of $\fs$ where $a_i,x_i$ spans $\fa,\fn$ respectively and $a\in \fa, x\in \fn,$ then from \cite[Equation 25]{lauret2010riccisolitonsolvmanifolds},
\begin{align}
\Ric(a,a)
&= -\frac{1}{2}\sum_{i}\|[a,a_i]\|^2
   -\tr\!\left(S(\ad_a|_{\fn})^2\right),
\label{Ric(a,a)} \\[6pt]
\Ric(a,x)
&= -\frac{1}{2}\sum_{i} g([a,a_i],[x,a_i])
   -\frac{1}{2}\tr\!\left((\ad_a|_{\fn})^t\,\ad_x|_{\fn}\right)
   -\frac{1}{2}g([H,a],x),
\label{Ric(a,x)} \\[6pt]
\Ric(x,x)
&= \frac{1}{4}\sum_{i,j} g([a_i,a_j],x)^2
   +\frac{1}{2}\sum_{i} g([\ad_{a_i}|_{\fn},(\ad_{a_i}|_{\fn})^t](x),x)\label{Ric(x,x)} \\
&\quad
   -\frac{1}{2}\sum_{i,j} g([x,x_i],x_j)^2
   +\frac{1}{4}\sum_{i,j} g([x_i,x_j],x)^2
   -g([H,x],x).\nonumber
\end{align}
This Ricci formula is essentially the same as \eqref{ricci}, but written with respect to the decomposition $\fa \oplus \fn$, whereas \eqref{ricci} is expressed using the splitting $[\fs,\fs] \oplus [\fs,\fs]^{\perp}$.
\begin{rmk}\label{rmk H}
Since $[\fs,\fs]\subseteq\fn,$ 
\begin{equation*}
 g(\ad_a^t (a_i),y)=g(a_i,[a,y])=0\quad \forall\, a,a_i\in \fa,y\in \fs.   
\end{equation*}
 This implies
\begin{equation*}
    \ad_a^t (a_i)=0 \hspace{2em}\forall\, 1\leq i\leq q.
\end{equation*}
\end{rmk}
We also use the following \cite[Lemma 4.7]{lauret2010riccisolitonsolvmanifolds}:
\begin{lemma}\label{Lauret lemma}
Let $S$ be a solvable Lie group with a left-invariant metric. Then, for any $a\in \fa,$ the following conditions are equivalent:
\begin{enumerate}[a)]
    \item $\ad_a^t$ is a derivation of $\fs.$
    \item $\ad_a$ is a normal operator (i.e., $[\ad_a,\ad_a^t]=0$).
\end{enumerate}

\end{lemma}
    Now we are ready to prove the structure theorem.
\begin{thm}\label{main solvable thm}
 Let $S$ be a non-flat unimodular solvable Lie group with left-invariant metric $g$ and Lie algebra $\fs=\fa\oplus \fn$ where $\fn$ is the nilradical of $\fs$ and suppose $\ad_a$ is normal for all $a\in\fa.$ Then $(S,g,X)$ is totally left-invariant quasi-Einstein with quasi-Einstein constant $\lambda$ if and only if the following conditions are satisfied:
 \begin{enumerate}[i)]
     \item $(N,g_1,X)$ is totally left-invariant quasi-Einstein with quasi-Einstein constant $\lambda$ where $N$ is any nilpotent Lie group with the Lie algebra $\fn$ and $g_1=g|_{\fn\times \fn},$
     \item $\fs$ is standard (i.e., $[\fa,\fa]=0$),
     \item $\fz(\fs)=\fz(\fn),$

 \item $g(a,a)=-\frac{1}{\lambda}\tr{S(\ad_a)^2}.$
 \end{enumerate}
\end{thm}
\begin{proof}
$(\Longleftarrow)$ We will prove that $r=\lambda I+ F$ on $S$. Let $a\in\fa, x\in \fn.$ Given $(N,g_1,X)$ is totally left-invariant quasi-Einstein. Theorem \ref{Lim} implies that $\cL_X g_1=0$ and Theorem \ref{solvable killing} says $X\in\fz(\fn)$ where $\fz(\fn)$ is one-dimensional. Because of $(iii)$, $X\in \fz(\fs)$ hence $\cL_Xg=0.$ Putting $[\fa,\fa]=0$ in (\ref{Ric(a,a)}) gives 
$$\Ric(a,a)=-\tr{S(\ad_a|_{\fn})^2}=\lambda g(a,a)$$
where the last equality is due to $(iv).$

To prove $\Ric(a,x)=0$:
Note that since $S$ is unimodular, the mean curvature vector $H=0.$ Since $[\fa,\fa]=0,$ the equation (\ref{Ric(a,x)}) implies that all we need to prove is $(\ad_a|_{\fn})^t \ad_x|_{\fn}=0.$ Let $\fn=\fn_1\oplus...\oplus \fn_r$ be an orthogonal decomposition such that $\fn=\fn_1\oplus [\fn,\fn]$, and $ [\fn,\fn]=\fn_2\oplus  [[\fn,\fn],\fn]$ so on. We claim that 
\begin{equation}\label{eqn A}
 \ad_a^t(\fn_i)\subset \fn_i \hspace{3em}\forall i.   
\end{equation}
Let $D\in Der(\fs),$ then $D([\fn,\fn])\subset [\fn,\fn],$ $D([\fn,[\fn,\fn]])\subset [\fn,[\fn,\fn]]$ and so on. In general,
\begin{equation*}
D(\fn_i\oplus...\oplus\fn_r)\subset \fn_i\oplus...\oplus\fn_r \hspace{3em}\forall i.
\end{equation*}
Let $u\in \fn_i$ and $v\in \fn_{i+1}\oplus ...\oplus \fn_{r}.$ Then 
\begin{equation*}
  g(\ad_a^t(u),v)=g(u,\ad_a(v))=0  
\end{equation*} since each $\fn_i$'s are orthogonal to each other. This proves our claim. Also if $x\in \fn,$ then 
\begin{equation}\label{eqn B}
 \ad_x(\fn_i)\subset [\fn,\fn_i]=\fn_{i+1}\oplus ...\oplus \fn_{r}.   
\end{equation} 
Then equation (\ref{eqn A}) and (\ref{eqn B}) together implies that 
\begin{equation*}
    \ad_a^t \ad_x(\fn_i)\subset \fn_{i+1}\oplus ...\oplus \fn_{r}.
\end{equation*}
Hence $\tr(\ad_a^t \ad_x)=\displaystyle\sum_{i} g(\ad_a^t \ad_x(\fn_i),\fn_i)=0.$ This proves that $\Ric{(a,x)}=0.$
To prove $\Ric{(x,x)}=g((\lambda I+F)(x),x),$ putting $[\fa,\fa]=0$ and Lemma \ref{Lauret lemma} in Ricci formula (\ref{Ric(x,x)}) gives
\begin{equation*}
    \Ric(x,x)=-\frac{1}{2}\sum g([x,x_i],x_j)^2+\frac{1}{4}g([x_i,x_j],x)^2 \newline
\end{equation*}
which is exactly $\Ric|_{\fn}.$ So $(i)$ implies that
\begin{equation*}
    \Ric{(x,x)}=\Ric|_{\fn}(x,x)=g_1((\lambda I+ F)(x),x)=g((\lambda I+ F)(x),x).
\end{equation*}
$(\Longrightarrow)$
To prove $(ii),$ since $\ad_a$ is normal, Lemma \ref{Lauret lemma} implies that
\begin{equation*}
    \ad_{a_{i}}^t[a_i,a_j]=[\ad_{a_{i}}^t(a_i),a_j]+[a_i,\ad_{a_{i}}^t(a_j)]\hspace{2em}\forall a_i,a_j\in \fa .
\end{equation*}
But Remark \ref{rmk H} implies that the right hand side is zero. So
\begin{equation*}
   0= g(\ad_{a_i}^t[a_i,a_j],a_j)=g([a_i,a_j],[a_i,a_j])
\end{equation*}
which gives that $[a_i,a_j]=0$, hence $(ii)$ is proved.

To prove $(i),$ consider (\ref{Ric(x,x)}). The first term vanishes due to $(ii)$ and second term vanishes since $\ad_a$ is normal. As a result, if $R_1$ denotes the Ricci operator on $\fn,$
\begin{equation*}
  R_1=r=\lambda I+ F
\end{equation*}
on $\fn,$ which proves $(i).$ Then Theorem \ref{nilpotent main theorem} implies that $X\in \fz(\fn)$ and $\fz(\fn)$ is one dimensional. But $X\in \fz(\fs)$ and $\fz(\fs)$ is also one-dimensional by Theorem \ref{solvable killing}. This proves $(iii).$

Finally from (\ref{Ric(a,a)}), we get $\lambda g(a,a)=-\tr{S(\ad_a|_{\fn})^2}$ which proves $(iv).$
\end{proof}
As an immediate corollary of the structure theorem, we obtain the following.  
\begin{cor}
Let $G$ be a non-flat unimodular solvable Lie group such that $\operatorname{ad}_a$ is normal for all $a \in \mathfrak{a}$. If $G$ admits a totally left-invariant quasi-Einstein metric, then its nilradical is isomorphic to a Heisenberg Lie algebra.  
\end{cor}

\subsection{The classification of totally left-invariant quasi-Einstein unimodular solvable Lie groups}

Now that we have a structure theorem, we classify the totally left-invariant quasi-Einstein unimodular solvable Lie groups such that $\ad_a$ is normal. As proved in \cite{Lim2020locally}, the only non-Abelian quasi-Einstein solvable Lie group in dimension three is the nilpotent Heisenberg group, $H_3.$ For higher dimensions, the results readily follow from Theorem \ref{main solvable thm}, as it suffices to identify the normal derivations of the Heisenberg Lie algebra.

\begin{rmk}\label{symmetric}
From \cite[Proposition 2.5]{Heber98}, a standard solvable Lie algebra with $\ad_a$ normal derivation for all $a\in \fa$  is isometric to the one obtained by just changing the Lie bracket into
\begin{equation*}
     [a,x]' =S(\ad_a)x, \,[x,y]' =[x,y], \forall a\in \fa, x,y \in \fn,
\end{equation*}
 and keeping the same metric $g.$ 
\end{rmk}
The above remark makes things easy because it implies that any totally left-invariant quasi-Einstein unimodular solvable Lie group with normal $\ad_a$ is isometric to another such Lie group that has $\ad_a$ symmetric for all $a\in \fa$.

The following theorem is the classification of operators of a totally left-invariant quasi-Einstein unimodular solvable Lie group that has normal $\ad_a$. This is achieved using a method similar to the one used by C. Will in \cite{cynthiaWill} to identify low-dimensional solvsolitons.
\begin{thm}\label{ada}
   Let $S$ be a non-flat unimodular solvable Lie group with $\ad_a$ normal for all $a\in \fa$. Then up to isometry, $S$ admits a totally left-invariant quasi-Einstein metric $g$ if and only if its Lie algebra $\fs= \fh_n \oplus _{\psi} \bR^k$ for some $0\leq k\leq s$ where $s=\frac{1}{2}(n-1),$ $\psi: \bR^k\rightarrow Der(\fh_n) $ is given by $a\mapsto \ad_a|_{\fh_n},$

 \[ \ad_a|_{\fh_n}= \begin{bmatrix}
   t_{1} &  &  & & &\\ 
   & -t_{1} &  & & &\\ 
   &  &  \ddots & & &\\ 
   &  &   & t_{s} & &\\
    &  &   & & -t_{s} & \\
     &  &   & &  & 0\\
 \end{bmatrix} \]
 with respect to the orthonormal basis $\{x_i,y_i,z|\, [x_i,y_j]=\delta_{ij}cz,\, 1\leq i,j\leq s\}$ of $\fh_n$ and $t_i\in \bR$ for all $1\leq i\leq s.$ 
\end{thm}
\begin{proof}
$(\Longleftarrow)$ Suppose $\fs= \fh_n \oplus _{\psi} \bR^k$ with the given adjoint action. Choose a left-invariant metric $g$ of $S$ such that $\displaystyle\sum^s_{i=1}t^2_i= \frac{c^2}{4}g(a,a)$ where $c\in\bR,[x_i,y_j]=\delta_{ij}cz$. To prove that $S$ is quasi-Einstein we have to verify Theorem \ref{main solvable thm}$i)-iv).$ Theorem \ref{heisenberg} implies that $i)$ is true.

To prove $ii),$ let $D=\ad_a^t.$ Since $\fa$ is Abelian, for any $a,a'\in \fa,$ we get $D[a,a']=0,$ and $D(a)=D(a')=0$ due to Remark \ref{rmk H}. Now,
\begin{equation*}
  D([a',x])=   \ad_a^t \ad_{a'} (x)=\ad_{a'} \ad_a^t (x)=[a',D(x)],
\end{equation*}
where the second equality is because $\ad_a$ is diagonal for all $a\in\fa.$ Using that $D(a')=0$ in the above equation, we get
\begin{equation*}
 D([a',x])=[D(a'),x]+[a',D(x)].
\end{equation*}
 So it is left to prove that $\ad_a^t|_{\fh_n}\in Der(\fh_{n}).$ Since $[x_i,y_j]=\delta_{ij}cz,$ from the given matrix $\ad_a|_{\fh_n}$ we get
\begin{equation*}
    [D(x_i),y_i]+[x_i,D(y_i)]=t_i[x_i,y_i]-t_i[x_i,y_i]=0
\end{equation*}
and $D([x_i,y_i])=D(z)=0$ for every $1\leq i\leq s.$ So $D$ satisfies the property of derivation for all the Lie brackets in $\fs$, hence it is a derivation of $\fs$. Proof of $iii)$ is immediate. To prove $iv),$ we choose an orthonormal basis $x_i,y_i,z$ for $1\leq i\leq s$ such that for any given $a\in \fa,$
\begin{equation*}
g(a,a)=\frac{4}{c^2}\displaystyle\sum^s_{i=1}t^2_i,
\end{equation*}
where $t_i=g(\ad_a(x_i),x_i)$ and $[x_i,y_j]=\delta_{ij}cz.$ But $a$ cannot be zero as it is in the orthogonal complement of $[\fa,\fa],$ hence $g(a,a)$ can't be zero. So there exists at least one $i,$ $1\leq i\leq s,$ such that $t_i\neq 0.$ Hence $[a,x_i]\neq 0$ and hence $a\notin \fz(\fs).$ So $\fz(\fs)=\fz(\fh_n)$ and $iv)$ is proved. Finally, $v)$ is true because according to the basis of $\fh_n$ we just chose, $\lambda=\Ric{(x_i,x_i)}=-\frac{1}{2}c^2$, which gives that
\begin{equation*}
    -\frac{1}{\lambda}\tr{S(\ad_a)^2}=\frac{4}{c^2}\sum^s_{i=1}t^2_i=g(a,a).
\end{equation*}

To prove the other direction, we will first prove that $\ad_a$ is the given diagonal matrix. Let the orthonormal basis of $\fh_n$ be $\{x_i,y_i,z| 1\leq i\leq s\}$ where $[x_i,y_j]=\delta_{ij}cz, c\in \bR\setminus\{0\}$ and the columns of $D=\ad_a|_{\fh_n}$ are ordered as $x_1,y_1,...,x_s,y_s, z$ from left to right.
\begin{equation*}
    g(D(x_i),x_i)=t_i,g(D(y_i),y_i)=s_i \,\,\,\forall\, 1\leq i\leq s,
\end{equation*}
for some $t_i,s_i\in \bR.$ Then
\begin{equation*}
    0=D(z)=[D(x_i),y_i]+[x_i,D(y_i)] =(t_i+s_i)z.
\end{equation*}
So $s_i=-t_i$ for all $1\leq i\leq s$ and the last column of $D$ is zero. By Remark \ref{symmetric}, $D$ is symmetric and hence the last row of $D$ is zero as well. Since $\bR^k$ is Abelian, the Jacobi identity implies that for any $a_i,a_j\in\bR^k,$ 
\begin{equation}\label{Abelian}
    \ad_{a_i}\circ \ad_{a_j}= \ad_{a_j}\circ \ad_{a_i}
\end{equation}
on $\fn.$ This implies that $\mathfrak{A}=\{\ad_a|_{\fh_n}: a\in \fa\}$ is an Abelian sub-algebra of the space of symmetric derivations of $\fh_n.$ So we can choose an orthonormal basis $\{x'_i,y'_i,z: [x'_i,y'_i]=z, 1\leq i\leq s \}$ of $\fh_n$ such that the elements of $\mathfrak{A}$ are all diagonal matrices as in \cite{cynthiaWill}. So $\ad_a|_{\fh_n}$ is diagonal with the diagonal entries $(t'_1,-t'_1,...,t'_s,-t'_s,0)$ with respect to the new basis.

To prove that $\fa$ in $\fs= \fh_n \oplus _{\psi} \fa$ is  $\bR^k$ for some $0\leq k\leq s,$ all we need to prove is $\dim(\fa)=k$ where $0\leq k\leq s$ as $\fa$ is Abelian by Theorem \ref{main solvable thm}. So we will prove that $\psi$ is an injection. Suppose $\psi(a)=\psi(b)$ for some basis vector $a,b$ of $\bR^k.$ Since $\psi$ is a linear map, $\psi(a-b)=0,$ hence $\ad_{a-b} =0$ on $\fs.$ But $\fs$ has one-dimensional center, hence $a-b \in \fz(\fs),$ so $a-b\in \fa \cap \fz(\fs).$ But $iv)$ implies that $\fz(\fs)=\fz(\fh_n)\subset \fh_n.$ Hence $a-b=0$ because $\fa\cap \fh_n =0.$ so $\psi$ is an injection, hence the dimension of $\fa$ is at most the dimension of the range
\begin{equation*}
 \{\ad_a|_{\fh_n}: a\in \fa\}\cong \{ (t_1,...,t_s): t_i\in \bR \hspace{1em}\forall \hspace{1em} 1\leq i\leq s
\}  
\end{equation*} which is $s.$
\end{proof}
\begin{rmk}
The numbers $t_i$ appearing in the above theorem are linear functionals 
$t_i \colon \mathbb{R}^k \to \mathbb{R}$, such that
\[
\mathrm{ad}_a\big|_{\mathfrak{h}_n} = 
\mathrm{diag}\!\left(t_1(a),\, -t_1(a),\, \ldots,\, t_s(a),\, -t_s(a),\, 0\right)
\]
for all $a \in \mathbb{R}^k, 1\leq  i\leq s$.
\end{rmk}
\begin{cor}
 Let $S$ be a non-flat unimodular solvable Lie group with $\ad_a$ normal for all $a\in\fa$. Then $S$ is totally left-invariant quasi-Einstein if and only if $S= \bR^k\ltimes_{\phi}H_n $ for some $0\leq k\leq \frac{1}{2}(n-1)$ where $H_n $ is the Heisenberg Lie group and $\phi: \bR^k\longrightarrow \Aut(H_n)$ such that differential of $(d\phi)_e$ is $\psi,$ where $\psi$ is the map given in Theorem \ref{ada}.
\end{cor}
\begin{proof}
The definition of the map $\phi$ at the Lie group level from the map $\psi$ at the Lie algebra level follows directly from \cite[Prop.~1.101]{MR1920389}.
\end{proof}
Theorem \ref{ada} in fact proves that there exist examples of non-flat unimodular solvable Lie group with totally left-invariant quasi-Einstein metric in every dimension greater than $2.$ In Table \ref{table2}, we describe the low dimensional examples.

\begin{table}[ht]
\caption{Classification of non-flat unimodular solvable Lie groups with normal $\ad_a$ and totally left-invariant quasi-Einstein metric of dimension $\leq 7$ up to isometry} 
\centering
\renewcommand{\arraystretch}{1.2} 
\begin{tabular}{c c c c c}
\hline\hline 
$dim(\fs)$ & $\fs$ & $\ad_a|_{\fh_n}$ & \textit{Constraints}  \\ 
\hline \vspace{1em}

3 & $\fh_3$ & -- & --  \\\vspace{1em}
4 & $\mathbb{R} a\oplus\fh_3$ & 
$\diag(\alpha, -\alpha,0)$ & $g(a,a)=(\tfrac{2\alpha}{c})^2$ \\ \vspace{1em}
5 & $\fh_5$ & -- & --  \\ \vspace{1em}
6 & $\mathbb{R} a\oplus\fh_5$ & 
$\diag(\alpha,\beta,-\alpha,-\beta,0)$ & $g(a,a)=4(\tfrac{\alpha^2+\beta^2}{c^2})$ \\ \vspace{1em}

7 & $\fh_7$ & -- & -- \\ 
7 & $\mathbb{R} a_1\oplus \mathbb{R} a_2 \oplus \fh_5$ & 
$\ad_{a_1}=\diag(\alpha,0,-\alpha,0,0),$

& $g(a_1,a_1)=(\tfrac{2\alpha}{c})^2, \;g(a_1,a_2)=0,$ \\ \vspace{2em}
 &  & $\ad_{a_2}=\diag(0,\beta, 0,-\beta,0),$ &  $g(a_2,a_2)=(\tfrac{2\beta}{c})^2$  \\
\hline
\end{tabular}
\label{table2}
\end{table}

 We do not yet understand the behavior in cases where $\ad_a$ is not normal. Even when $\dim \mathfrak{a} = 1$, it remains an interesting open question whether there exist solvable Lie groups with non-normal $\ad_a$ that admit quasi-Einstein metrics.

\begin{rmk}
A Riemannian solvmanifold is a homogeneous space of a connected solvable Lie group. Any simply connected solvmanifold is isometric to a solvable Lie group $S$ equipped with a left-invariant metric (see \cite{MR936815}).  
Since our structure theorem is stated up to isometry, they can also be interpreted as structure theorem for simply connected solvmanifolds that admit a totally left-invariant quasi-Einstein metric with $\ad_a$ normal, where the corresponding solvable Lie group is unimodular.
\end{rmk}
 
We conclude by discussing a case that remains unexplored in our analysis.
\subsection{The case when X is not left-invariant } \label{subsec B} Finding quasi-Einstein metrics on unimodular solvable Lie groups becomes significantly more challenging when the vector field $X$ is not assumed to be left-invariant. Even in the simplest case of the three-dimensional Heisenberg group, it remains unknown whether such solutions exist when $X$ lacks left-invariance. However, in the event that such solutions do exist, certain properties of these solutions are known to us. 
\begin{prop}
 Suppose  $G$ is a Lie group that satisfies the equation $\mRic{}=\lambda g$ where $X$ is not left-invariant, $g$ is left-invariant with both positive and negative directions of Ricci curvature, and the Lie algebra of $G$ admits an orthonormal basis $e_1,...,e_n$ such that $g([e_i,e_j],e_i)=0$ for all $1\leq i\leq n.$ Then $m\lambda < 0. $   
\end{prop}
\begin{proof}
Suppose $X=\sum _i f_i e_i$ where $e_i$'s form a basis as given above. Since $X$ is not left-invariant, we can assume that $f_i$ is non-constant smooth function on $M$ for some $1\leq i\leq n.$ Then

\begin{align*}
    (\cL_{X}g)(u,v) &= D_{X}g(u,v) - g(\cL_X u, v) - g(u, \cL_X v) \\
    &= -g([X,u],v) - g(u,[X,v]) \\
    &= -\sum_{i} g([f_i e_i, u], v) - g(u, [f_i e_i, v]) \\
    &= -\sum_{i} f_i g([e_i,u], v) + D_u(f_i) g(e_i, v) 
       - f_i g(u, [e_i, v]) + D_v(f_i) g(e_i, u).
\end{align*}
Suppose $u=v=e_j,$ where $e_j$ is some basis vector,
\begin{equation*}
     (L_{X}g)(e_j,e_j) 
    = -\sum_{i} 2 f_i g([e_i,e_j], e_j) + 2D_u(f_j) g(e_j, e_j) = \sum_{i} 2D_u(f_j).   
\end{equation*}
Then, $\mRic{}=\lambda g$ implies that
 \begin{equation*}
      \sum_{i} 2D_u(f_j)-\frac{1}{m}f^2_j = -\Ric(e_j,e_j) + \lambda,
 \end{equation*}
for every $1\leq j\leq n.$ Then Lemma 2.3 of \cite{Lim2020locally} implies that $m(-\Ric(e_j,e_j) + \lambda)<0.$ So, $m>0$ implies that $\lambda < \Ric(e_j,e_j),$ hence $\lambda <0$ since $G$ has some vectors that are Ricci negative. Similarly,  $m<0$ will imply that $\lambda >0.$ 
\end{proof}
Next, we prove a lemma that will be used in a corollary of this theorem.
\begin{lemma}\label{lemma A1}
Any nilpotent Lie group $(N,g)$ with Lie algebra $\fn$ admits an orthonormal basis $\{e_1,...,e_n\}$ such that $g([e_i,e_j],e_i)=0$ for any $1\leq i,j\leq n.$
\end{lemma}
\begin{proof}
 Suppose $\fn_1=[\fn,\fn]^{\perp}$ inside $\fn,$ $\fn_2= [\fn,[\fn,\fn]]^{\perp}$ inside $[\fn,\fn]$ and so on. Since $\fn$ is nilpotent, there exists a finite real number $r$ such that $\fn=\fn_1\oplus ...\oplus \fn_r$ is an orthogonal decomposition with $[\fn,\fn]=\fn_2\oplus ...\oplus \fn_r, [\fn,[\fn,\fn]]=\fn_3\oplus ...\oplus \fn_r$ and so on. For $1\leq i\leq r,$ let $\{e_{i1},e_{i2},...,e_{in_{{i}}}\}$ be the orthonormal basis that span each $\fn_i$ with dimension $n_{i}.$ Then $g([e_{ik},e_{il}],e_{ik})=0$ since $[e_{ik},e_{il}]\in \fn_{i+1}\oplus...\oplus \fn_r$ by construction for any $1\leq i\leq r$ and $1\leq k,l\leq n_{i}.$ Also, suppose $i<j,$ then $[e_{ik},e_{jl}]\in \fn_{j+1}\oplus...\oplus \fn_r$ by construction. So $g([e_{ik},e_{jl}],e_{ik})=g([e_{ik},e_{jl}],e_{jl})=0.$ Hence $\{e_{i1},e_{i2},...,e_{in_{{i}}}, 1\leq i\leq r\}$ is an orthonormal basis of $\fn$ that satisfies the desired property.
\end{proof}
\begin{cor}
Suppose a nilpotent Lie group satisfies the equation $\mRic{}=\lambda g$ where $X$ is not left-invariant and $g$ is left-invariant, then $m\lambda < 0. $    
\end{cor}
\begin{proof}
 From lemma \ref{lemma A1}, we have a basis for the Lie algebra such that $g([e_i,e_j],e_i)=0.$ 
 \end{proof}
When $\lambda$ is non-negative, regardless of whether $X$ becomes left-invariant, the following statements can be made.
\begin{prop}
Suppose $(M,g,X)$ admits an $m$-quasi-Einstein metric where $M$ is a homogeneous space. If $m>0,$ then:
\begin{enumerate}[a)]
    \item If $\lambda >0,$ then $M$ is compact.
    \item If $\lambda= 0$ then $M$ splits
isometrically as $N\times \bR^k$ where $N$ is a compact homogeneous space.
\end{enumerate}
\end{prop}
\begin{proof}
 Suppose $m $ and $ \lambda$ are strictly positive. Then by \cite[Theorem 1.1 ]{Limoncu2010-nr}, we have that $M$ is compact.
 
 When $\lambda=0,$ by \cite[Theorem 2]{2019LMaPh.109..661K}, $M$ splits
isometrically as $N\times \bR^k$ where $N$ is a complete Riemannian manifold without a line. Then, applying the argument in the proof of Theorem 7.3.16 of \cite{petersen2016riemannian}, one can conclude that $N$ must be compact.
\end{proof}

\section*{Declarations}

\noindent\textit{\textbf{Funding:} No funding was received for conducting this research.}
\medskip

\bibliographystyle{amsalpha}
\bibliography{references1} 

@article{Lim2020locally,
volume = {22},
year = {2022},
title = {Locally homogeneous non-gradient quasi-{E}instein 3-manifolds},
author = {Lim, Alice},
address = {Berlin},
copyright = {2022 Walter de Gruyter GmbH, Berlin/Boston},
issn = {1615-715X},
journal = {Advances in geometry},
number = {1},
pages = {79-93},
publisher = {De Gruyter},
}

@article{Kir62,
volume = {17},
year = {1962},
title = {UNITARY REPRESENTATIONS OF NILPOTENT {L}IE GROUPS},
author = {Kirillov, A A},
issn = {0036-0279},
journal = {Russian mathematical surveys},
number = {4},
pages = {53-104},
publisher = {IOP Publishing},
}

@article{DottiMiatello1982,
author = {Dotti Miatello, Isabel},
journal = {Mathematische Zeitschrift},
pages = {257-264},
title = {{R}icci Curvature of Left Invariant Metrics on Solvable Unimodular {L}ie Groups},
url = {http://eudml.org/doc/173186},
volume = {180},
year = {1982},
}

@article{MILNOR1976293,
title = {Curvatures of left invariant metrics on {L}ie groups},
journal = {Advances in Mathematics},
volume = {21},
number = {3},
pages = {293-329},
year = {1976},
issn = {0001-8708},
doi = {https://doi.org/10.1016/S0001-8708(76)80002-3},
url = {https://www.sciencedirect.com/science/article/pii/S0001870876800023},
author = {John Milnor},

}

@article{lauret2010riccisolitonsolvmanifolds,
volume = {2011},
year = {2011},
title = {{R}icci soliton solvmanifolds},
author = {Lauret, Jorge},
issn = {0075-4102},
journal = {Journal für die reine und angewandte Mathematik},
number = {650},
pages = {1-21},
publisher = {Walter de Gruyter GmbH & Co. KG},
}

@Book{Rag72,
  author    = {Raghunathan, M. S},
  publisher = {Springer-Verlag},
  title     = {Discrete subgroups of {L}ie groups},
  year      = {1972},
  isbn      = {0387057498},

}

@article{Heber98,
  author  = {Heber, J },
  journal = {Invent math},
  title   = {  Noncompact homogeneous {E}instein spaces},
  year    = {1998},
  pages   = {279–352 },
  volume  = {133},
    URL = { https://doi.org/10.1007/s002220050247} 
    }

@incollection{GOZE2000615,
title = {Nilpotent and solvable {L}ie algebras},
editor = {M. Hazewinkel},
series = {Handbook of Algebra},
publisher = {North-Holland},
volume = {2},
pages = {615-663},
year = {2000},
issn = {1570-7954},
doi = {https://doi.org/10.1016/S1570-7954(00)80040-6},
url = {https://www.sciencedirect.com/science/article/pii/S1570795400800406},
author = {Michel Goze and Yusupdjan Khakimdjanov},

}

@article {cochran2024killingfieldscompactmquasieinstein,
    AUTHOR = {Cochran, Eric},
     TITLE = {Killing fields on compact {$m$}-quasi-{E}instein manifolds},
   JOURNAL = {Proc. Amer. Math. Soc.},
  FJOURNAL = {Proceedings of the American Mathematical Society},
    VOLUME = {153},
      YEAR = {2025},
    NUMBER = {2},
     PAGES = {841--849},
      ISSN = {0002-9939,1088-6826},
   MRCLASS = {53C25},
  MRNUMBER = {4852803},
MRREVIEWER = {Chandan\ Kumar\ Mondal},
       DOI = {10.1090/proc/17063},
       URL = {https://doi.org/10.1090/proc/17063},
}

@ARTICLE{chen-Liang-Zhu,
  title    = "Non-trivial m-quasi-{E}instein metrics on simple {L}ie groups",
  author   = "Chen, Zhiqi and Liang, Ke and Zhu, Fuhai",
  journal  = "Annali di Matematica Pura ed Applicata (1923 -)",
  volume   =  195,
  number   =  4,
  pages    = "1093--1109",
  month    =  aug,
  year     =  2016
}

@ARTICLE{Kunduri2013-lj,
  title    = "Classification of {Near-Horizon} Geometries of Extremal Black
              Holes",
  author   = "Kunduri, Hari K and Lucietti, James",
  journal  = "Living Reviews in Relativity",
  volume   =  16,
  number   =  1,
  pages    = "8",
  month    =  sep,
  year     =  2013
}

@ARTICLE{Bahuaud2025-vb,
  title    = "Extreme 5-dimensional black holes with {SU(2)-symmetric} horizons",
  author   = "Bahuaud, Eric and Gunasekaran, Sharmila and Kunduri, Hari K and
              Woolgar, Eric",
  journal  = "Journal of High Energy Physics",
  volume   =  2025,
  number   =  3,
  pages    = "218",
  month    =  mar,
  year     =  2025
}

@article{cynthiaWill,
    author = {Will, Cynthia},
    title = "{The space of solvsolitons in low dimensions}",
    journal = {Annals of Global Analysis and Geometry
},
    volume = {40},
    pages = {291–309},
    year = {2011},
    month = {03},
    url = {https://doi.org/10.1007/s10455-011-9258-0},

}

@article{Okumura1962SomeRO,
  title={Some remarks on space with a certain contact structure},
  author={Masafumi Okumura},
  journal={Tohoku Mathematical Journal},
  year={1962},
  volume={14},
  pages={135-145},
  url={https://api.semanticscholar.org/CorpusID:122084399}
}

@ARTICLE{Limoncu2010-nr,
  title    = "Modifications of the {R}icci tensor and applications",
  author   = "Limoncu, Murat",
  journal  = "Archiv der Mathematik",
  volume   =  95,
  number   =  2,
  pages    = "191--199",
  month    =  aug,
  year     =  2010
}

@ARTICLE{2019LMaPh.109..661K,
       author = {{Khuri}, Marcus and {Woolgar}, Eric and {Wylie}, William},
        title = "{New restrictions on the topology of extreme black holes}",
      journal = {Letters in Mathematical Physics},
         year = 2019,
        month = mar,
       volume = {109},
       number = {3},
        pages = {661-673},
          doi = {10.1007/s11005-018-1121-9},
archivePrefix = {arXiv},
       eprint = {1804.01220},
 primaryClass = {hep-th},
       adsurl = {https://ui.adsabs.harvard.edu/abs/2019LMaPh.109..661K}
}

@book {MR1920389,
    AUTHOR = {Knapp, Anthony W.},
     TITLE = {Lie groups beyond an introduction},
    SERIES = {Progress in Mathematics},
    VOLUME = {140},
   EDITION = {Second},
 PUBLISHER = {Birkh\"auser Boston, Inc., Boston, MA},
      YEAR = {2002},
     PAGES = {xviii+812},
      ISBN = {0-8176-4259-5},
   MRCLASS = {22-01},
  MRNUMBER = {1920389},
}

@article {MR936815,
    AUTHOR = {Gordon, Carolyn S. and Wilson, Edward N.},
     TITLE = {Isometry groups of {R}iemannian solvmanifolds},
   JOURNAL = {Trans. Amer. Math. Soc.},
  FJOURNAL = {Transactions of the American Mathematical Society},
    VOLUME = {307},
      YEAR = {1988},
    NUMBER = {1},
     PAGES = {245--269},
      ISSN = {0002-9947,1088-6850},
   MRCLASS = {53C30},
  MRNUMBER = {936815},
MRREVIEWER = {A.\ L.\ Onishchik},
       DOI = {10.2307/2000761},
       URL = {https://doi-org.libezproxy2.syr.edu/10.2307/2000761},
}

@article{Ghosh2020mQuasiEinsteinMS,
  title={m-Quasi-Einstein Metrics Satisfying Certain Conditions on the Potential Vector Field},
  author={Amalendu Ghosh},
  journal={Mediterranean Journal of Mathematics},
  year={2020},
  volume={17},
  url={https://api.semanticscholar.org/CorpusID:225451294}
}

@book{petersen2016riemannian,
  title={Riemannian Geometry},
  author={Petersen, Peter},
  year={2016},
  publisher={Springer},
  address={Cham},
  series={Graduate Texts in Mathematics},
  isbn={978-3-319-26654-1}
}

@article{case2011rigidity,
  title={Rigidity of quasi-{E}instein metrics},
  author={Case, J. and Shu, Y. and Wei, G.},
  journal={Differential Geometry and Applications},
  volume={29},
  pages={93--100},
  year={2011}
}

@ARTICLE{Costa-Filho2024-dk,
  title    = "A note on closed quasi-{E}instein manifolds",
  author   = "Costa-Filho, Wagner Oliveira",
  journal  = "Analysis and Mathematical Physics",
  volume   =  14,
  number   =  5,
  pages    = "107",
  month    =  sep,
  year     =  2024
}

@article{PETERSEN2022101929,
title = {Rigidity of homogeneous gradient soliton metrics and related equations},
journal = {Differential Geometry and its Applications},
volume = {84},
pages = {101929},
year = {2022},
issn = {0926-2245},
doi = {https://doi.org/10.1016/j.difgeo.2022.101929},
url = {https://www.sciencedirect.com/science/article/pii/S0926224522000821},
author = {Peter Petersen and William Wylie}
}

@article{https://doi.org/10.1112/blms/bdu103,
author = {Lafuente, Ramiro A.},
title = {On homogeneous warped product {E}instein metrics},
journal = {Bulletin of the London Mathematical Society},
volume = {47},
number = {1},
pages = {118-126},
doi = {https://doi.org/10.1112/blms/bdu103},
url = {https://londmathsoc.onlinelibrary.wiley.com/doi/abs/10.1112/blms/bdu103},
eprint = {https://londmathsoc.onlinelibrary.wiley.com/doi/pdf/10.1112/blms/bdu103},
year = {2015}
}

@article {MR39734,
    AUTHOR = {Malcev, A. I.},
     TITLE = {On a class of homogeneous spaces},
   JOURNAL = {Amer. Math. Soc. Translation},
  FJOURNAL = {Amer. Math. Soc. Translation},
    VOLUME = {1951},
      YEAR = {1951},
    NUMBER = {39},
     PAGES = {33},
   MRCLASS = {20.0X},
  MRNUMBER = {39734},
}

@article {MR1825405,
    AUTHOR = {Lauret, Jorge},
     TITLE = {Ricci soliton homogeneous nilmanifolds},
   JOURNAL = {Math. Ann.},
  FJOURNAL = {Mathematische Annalen},
    VOLUME = {319},
      YEAR = {2001},
    NUMBER = {4},
     PAGES = {715--733},
      ISSN = {0025-5831,1432-1807},
   MRCLASS = {53C25 (22E25 53C30)},
  MRNUMBER = {1825405},
MRREVIEWER = {J\"urgen\ Berndt},
       DOI = {10.1007/PL00004456},
       URL = {https://doi.org/10.1007/PL00004456},
}

@article {boyer,
    AUTHOR = {Boyer, Charles P.},
     TITLE = {The {S}asakian geometry of the {H}eisenberg group},
   JOURNAL = {Bull. Math. Soc. Sci. Math. Roumanie (N.S.)},
  FJOURNAL = {Bulletin Math\'ematique de la Soci\'et\'e{} des Sciences
              Math\'ematiques de Roumanie. Nouvelle S\'erie},
    VOLUME = {52(100)},
      YEAR = {2009},
    NUMBER = {3},
     PAGES = {251--262},
      ISSN = {1220-3874,2065-0264},
   MRCLASS = {53C25 (22E25)},
  MRNUMBER = {2554644},
}

@article{VUKMIROVIC201572,
title = {Classification of left-invariant metrics on the Heisenberg group},
journal = {Journal of Geometry and Physics},
volume = {94},
pages = {72-80},
year = {2015},
issn = {0393-0440},
doi = {https://doi.org/10.1016/j.geomphys.2015.01.005},
url = {https://www.sciencedirect.com/science/article/pii/S0393044015000066},
author = {Srdjan Vukmirović}
}

\end{document}